\documentclass[12pt,reqno]{amsart}
\usepackage{latexsym}
\usepackage{enumitem}
\setlist{leftmargin=5.5mm}

\usepackage{amsmath}
\usepackage{amscd}
\usepackage{amssymb}
\usepackage{dsfont}
\usepackage{braket}
\usepackage{empheq}
\usepackage{mathrsfs}
\usepackage{comment}
\usepackage{color}
\usepackage{soul}
\usepackage{framed}
\usepackage[normalem]{ulem}
\usepackage{graphicx}
\usepackage{subfigure}
\usepackage[colorlinks]{hyperref}
\usepackage{url}

\newtheorem{theorem}{Theorem}[section]
\newtheorem{corollary}{Corollary}[theorem]
\newtheorem{definition}{Definition}[section]

\newtheorem{example}[theorem]{Example}

\newtheorem{lemma}[theorem]{Lemma}

\newtheorem{proposition}[theorem]{Proposition}
\newtheorem{remark}[theorem]{Remark}

\newtheorem{assump}{Assumption}

\numberwithin{equation}{section} 

\DeclareMathOperator{\diff}{d\!}
\newcommand{\norm}[1]{\left|#1\right|}   

\begin{document}

\title[The Burgers' equation with stochastic transport ]{The Burgers' equation with stochastic transport: shock formation, local and global existence of smooth solutions}

\author[D. Alonso-Or\'an, A. Bethencourt de Le\'on and S.Takao]{Diego Alonso-Or\'an, Aythami Bethencourt de Le\'on and So Takao}

\address{DA: Instituto de Ciencias Matem\'aticas CSIC-UAM-UC3M-UCM, 28049 Madrid, Spain.}
\email{DA: diego.alonso@icmat.es}

\address{AB, ST: Department of Mathematics, Imperial College, London SW7 2AZ, UK.}
\email{AB: ab1113@ic.ac.uk \ \  ST:st4312@ic.ac.uk}

\maketitle

\begin{abstract}
In this work, we examine the solution properties of the Burgers' equation with stochastic transport. First, we prove results on the formation of shocks in the stochastic equation and then obtain a stochastic Rankine-Hugoniot condition that the shocks satisfy. Next, we establish the local existence and uniqueness of smooth solutions in the inviscid case and construct a blow-up criterion. Finally, in the viscous case, we prove global existence and uniqueness of smooth solutions.
\end{abstract}



\setcounter{tocdepth}{1}

\section{Introduction}  \label{section1}

We prove the well-posedness of a stochastic Burgers' equation of the form
\begin{align}
    \diff u(t,x) + u(t,x)\partial_x u (t,x) \,\diff t + \sum_{k=1}^\infty \xi_k(x) \partial_x u(t,x) \circ \diff W_t^k = \nu \partial_{xx} u(t,x) \diff t, \label{eq0}
\end{align}
where $x \in \mathbb T$ or $\mathbb R,$ $\nu \geq 0$ is constant, $\{W_t^k\}_{k \in \mathbb N}$ is a countable set of independent Brownian motions, $\{\xi_k(\cdot)\}_{k \in \mathbb N}$ is a countable set of prescribed functions depending only on the spatial variable, and $\circ$ means that the stochastic integral is interpreted in the Stratonovich sense. If the set $\{\xi_k(\cdot)\}_{k \in \mathbb N}$ forms a basis of some separable Hilbert space $\mathcal H$ (for example $L^2(\mathbb T)$), then the process $\diff W := \sum_{k=1}^\infty \xi_k(x) \circ \diff W_t^k$ is a {\em cylindrical Wiener process} on $\mathcal H$, generalising the notion of a standard Wiener process to infinite dimensions.

The multiplicative noise in \eqref{eq0} makes the transport velocity stochastic, which allows the Burgers' equation to retain the form of a transport equation $\partial_t u + \tilde{u} \,\partial_x u = 0,$ where $\tilde{u}(t,x) := u(t,x) + \dot{W}$ is a {\em stochastic vector field} with noise $\dot{W}$ that is smooth in space and rough in time. Compared with the well-studied Burgers' equation with additive noise, where the noise appears as an external random forcing, this type of noise arises by taking the diffusive limit of the Lagrangian flow map regarded as a composition of a slow mean flow and a rapidly fluctuating one \cite{homogenization2017}. In several recent works, this type of noise, which we call {\em stochastic transport}, has been used to stochastically parametrise unresolved scales in fluid models while retaining the essential physics of the system \cite{holm2015,cotter2018a,cotter2018b}. On the other hand, it has also been shown to have a regularising effect on certain PDEs that are ill-posed \cite{flandoli2010,flandoli2011,flandoli2013,gess2017}. Therefore, it is of interest to investigate how the stochastic transport in \eqref{eq0} affects the Burgers' equation, which in the inviscid case $\nu=0$ is a prototypical model for {\em shock formation}.  In particular, we ask whether this noise can prevent the system from developing shocks or, on the contrary, produce new shocks. We also ask whether this system is well-posed or not. In this paper, we will show that:\\
\begin{enumerate}
    \item For $\nu = 0$, equation \eqref{eq0} has a unique solution of class $H^s$ for $s > 2$ until some stopping time $\tau > 0$.
    \item However, shock formation cannot be avoided a.s. in the case $\xi(x) = \alpha x + \beta$ and for a broader class of $\{\xi_k(\cdot)\}_{k \in \mathbb N}$, we can prove that it occurs in expectation.
    \item For $\nu > 0$, we have global existence and uniqueness in $H^s$ for $s>2$.
\end{enumerate}
On top of this, we prove a continuation criterion for the inviscid equation ($\nu=0$), which generalises the result for the deterministic case. The above results are not immediately evident for reasons we will discuss below. Although we cannot prove this here, we believe that shocks in Burgers' equation are too robust and ubiquitous to be prevented by noise, regardless of what $\{\xi_k(\cdot)\}_{k \in \mathbb N}$ is chosen. Our results provide rigorous evidence to support this claim.
 
 The question of whether noise can regularise PDEs is not new. In finite dimensions, it is well-known that additive noise can restore the well-posedness of ODEs whose vector fields are merely bounded and measurable (see \cite{veretennikov1981}). For PDEs, a general result is not known; however, there has been a significant effort in recent years to generalise this celebrated result to PDEs. In a remarkable paper, Flandoli, Gubinelli, and Priola \cite{flandoli2010} demonstrated that the linear transport equation $\partial_t u + b(t,x) \cdot \nabla u = 0$, which is ill-posed if $b$ is sufficiently irregular, can recover existence and uniqueness of $L^\infty$ solutions that is strong in the 
 probabilistic sense, by the addition of a ``simple'' transport noise,
 \begin{align}
     \diff u + b(t,x) \cdot \nabla u \diff t + \nabla u \circ \diff W_t = 0,
 \end{align}
 where the drift $b$ is bounded, measurable, H\"older continuous, and satisfies an integrability condition on the divergence $\nabla \cdot b \in L^p([0,T] \times \mathbb R^d)$. In a subsequent paper \cite{flandoli2013}, the same noise was shown to retain some regularity of the initial condition, thus restoring well-posedness of strong solutions, and a selection principle based on taking the zero-noise limit as opposed to the inviscid limit was considered in \cite{attanasio2009}. 

However, for nonlinear transport equations such as Burgers', the same type of noise $\diff u + u\,\partial_xu \diff t + \partial_x u \circ \diff W_t = 0$ does not help, since a simple change of variables $v(t,x) := u(t,x-W_t)$ will lead us back to the original equation $\partial_t v + v \,\partial_x v = 0$. Hence, if noise were to prevent shock formation, a more general class could be required, such as the cylindrical transport noise $\sum_{k=1}^\infty \xi_k(x) \partial_x u \circ \diff W_t^k$ that we consider in this paper. In \cite{flandoli2011} and \cite{delarue2014}, it was shown that collapse in Lagrangian point particle solutions of certain nonlinear PDEs (point vortices in 2D Euler and point charges in the Vlasov-Poisson system), can be prevented by this cylindrical transport noise with $\xi_k(x)$ satisfying a certain hypoellipticity condition, thus providing hope for regularisation of nonlinear transport equation by noise.
More recently, Gess and Maurelli \cite{gess2017} showed that adding a simple stochastic transport term into a nonlinear transport equation
\begin{align}
    \diff u + b(x,u(t,x)) \nabla u \diff t + \nabla u \circ \diff W_t = 0,
\end{align}
which in the deterministic case admits non-unique entropy solutions for sufficiently irregular $b$, can restore uniqueness of entropy solutions, providing a first example of a nonlinear transport equation that becomes well-posed when adding a suitable noise.

We should now stress the difference between the present work and previous works. First, we acknowledge that in Flandoli \cite{BookOfDoom}, Chapter \textcolor{red}{5.1.4}, it is argued that shock formation does not occur even with the most general cylindrical transport noise, by writing the characteristic equation as an {\em It\^o} SDE
\begin{align}
    X_t = X_0 + u(0,X_0) t + \sum_{k=1}^\infty \int^t_0 \xi_k(X_s) \diff W_s^k,
\end{align}
which is a martingale perturbation of straight lines that will cross without noise. Thus, using the property that a martingale $M_t$ grows slower than $t$ almost surely as $t \rightarrow \infty$, it is shown that the characteristics cross almost surely. However, the characteristic equation for the system \eqref{eq0} is in fact a {\em Stratonovich} SDE,
\begin{align}
    X_t = X_0 + u(0,X_0) t + \sum_{k=1}^\infty \int^t_0 \xi_k(X_s) \circ \diff W_s^k, \label{char1}
\end{align}
and therefore Flandoli's argument can be applied to the martingale term, but not to the additional drift term, which may disrupt shock formation.
The techniques we use here apply to Stratonovich equations; however, due to the difficulty caused by the additional drift term, we were only able to prove that the characteristics cross almost surely in the very particular case $\xi(x) = \alpha x + \beta$, leaving the general case open for future investigation. By using a different strategy, where instead we look at how the slope $\partial_x u$ evolves along a characteristic \eqref{char1}, we manage to show that for a wider class of $\{\xi_k(\cdot)\}_{k=1}^\infty$ such that the infinite sum \mbox{$\sum_{k\in\mathbb N} ((\partial_x \xi_k)^2 - \xi_k \partial_{xx}\xi_k)$} is pointwise bounded, we have that
\begin{itemize}
    \item if $\partial_x u(0,X_0) > 0$, then $\partial_x u(t,X_t) < \infty$ almost surely for all $t>0$ and\\
    \item if $\partial_x u(0,X_0)$ is sufficiently negative, then there exists $0 < t_* < \infty$ such that\\
    $\lim_{t \rightarrow t_*} \mathbb E[\partial_x u (t,X_t)] = -\infty$.
\end{itemize}
In summary, shock formation occurs in expectation if the initial profile has a sufficiently negative slope and no new shocks can form from a positive slope.

We finally address the question of well-posedness. We will prove that by choosing a sufficiently regular initial condition, equation \eqref{eq0} admits a unique local solution that is smooth enough, such that the arguments employed in the previous section on shock formation are valid (in fact, we show this for a noise of the type $Qu \circ \diff W_t,$ where $Qu = a(x) \partial_x u + b(x) u,$ which generalises the one considered in \eqref{eq0}). For Burgers' equation with additive space-time white noise, however, there have been many previous works showing well-posedness \cite{StoBurger1994,2StoBurger1994,3StoBurger1994,4StoBurger1994}. The techniques used in these works are primarily based on reformulating the equations by a change of variable or by studying its linear part. The main difference in our work is that the multiplicative noise we consider depends on the solution and its gradient. Therefore, the effect of the noise hinges on its spatial gradient and the solution, giving rise to several complications. For instance, when deriving a priori estimates, certain high order terms appear, which need to be treated carefully. Recently, the same type of multiplicative noise has been treated for the Euler equation \cite{stochEuler2017,stoEuler2018} and the Boussinesq system \cite{stochBouss2018}, whose techniques we follow closely in our proof. We note that the well-posedness analysis of a more general stochastic conservation law, which includes the inviscid stochastic Burgers' equation as a special case, has also been considered, for instance in \cite{friz2016stochastic, gess2017stochastic,funaki2019uniqueness}. However, these works deal with the well-posedness analysis of weak kinetic and entropy solutions, in contrast to classical solutions, which we consider here.
There is also the recent work \cite{hocquet2018} showing the local well-posedness of weak solutions in the viscous Burgers' equation ($\nu > 0$) driven by rough paths in the transport velocity. An important contribution of this paper is showing the global well-posedness of strong solutions in the viscous case by proving that the maximum principle is retained under perturbation by stochastic transport of type $\xi(x) \partial_x u(t,x) \circ \diff W_t$.

\subsection{Main results}
Let us state here the main results of the article:
\begin{theorem}[Shock formation in the stochastic Burgers' equation] \label{theorem-1}
In the following, we use the notation $\psi(x) := \frac12\sum_{k=1}^\infty \left( (\partial_x \xi_k(x))^2 - \xi_k(x) \partial_{xx} \xi_k(x)\right)$. The main results regarding shock formation in \eqref{eq0} are as follows:
\begin{enumerate}
    \item Let $\xi_1(x) = \alpha x + \beta$, $x \in \mathbb R$ and $\xi_k \equiv 0$ for $k=2,3,\ldots$ and assume that $u(0,x)$ has a negative slope. Then, there exists two characteristics satisfying \eqref{char1} with different initial conditions that cross in finite time almost surely.
    \\
    \item Let $X_t$ be a characteristic solving \eqref{char1} with $\{\xi_k(\cdot)\}_{k \in \mathbb N}$ satisfying the conditions in Assumption \ref{xi-assump} below and let $\partial_x u(0,X_0) \geq 0$. Then, if $\psi(x) < \infty$ for all $x \in \mathbb T$ or $\mathbb R$, we have that $\partial_t u(t,X_t) < \infty$ almost surely for all $t>0$.
    \\
    \item Again, let $X_t$ be a characteristic solving \eqref{char1} with $\{\xi_k(\cdot)\}_{k \in \mathbb N}$ satisfying the conditions in Assumption \ref{xi-assump} and let $\partial_x u(0,X_0) < 0$. Also assume that $\partial_x u(0,X_0) < \psi(x)$ for all $x \in \mathbb T$ or $\mathbb R$. Then there exists $0<t_*<\infty$ such that $\lim_{t \rightarrow t_*} \mathbb E [\partial_x u(t,X_t)] = - \infty$.
\end{enumerate}
\end{theorem}

\begin{theorem}[Stochastic Rankine-Hugoniot condition] \label{theorem-2}
The curve of discontinuity \\$(t,s(t)) \in [0,\infty) \times \mathbb T$ (or $\mathbb R$) of the stochastic Burgers' equation \eqref{eq0} satisfies the following:
\begin{align}
\diff s_t = \frac12 \left[(u_-(t,s(t)) + u_+(t,s(t))\right]\,\diff t + \sum_{k=1}^\infty \xi_k(s(t)) \circ \diff W_t^k,
\end{align}
where $u_\pm (t,s(t)) := \lim_{x \rightarrow s(t)^\pm} u(t,x)$ are the left and right limits of $u$.
\end{theorem}

\begin{theorem}[Well-posedness in the inviscid case] \label{theorem-3}
Let $u_{0}\in H^{s} (\mathbb{T})$ for some $s > 2$ fixed. Then there exists a pathwise unique $H^s$-maximal solution $(\tau_{max},u)$ of the 1D Burgers' equation \eqref{eq0} in the sense of Definition \ref{def:maximal} with initial datum $u_0$. 
Moreover, either $\tau_{max}=\infty$ or $\lim\sup_{t \rightarrow \tau_{max}} \norm{u(t)}_{H^{s}} = \infty,$ a.s.
\end{theorem}

\begin{theorem}[Global well-posedness in the viscous case] \label{theorem-4}
Let $u_{0} \in H^{s} (\mathbb{T})$ for some $s> 2$ fixed. Then there exists a pathwise unique maximal global $H^s$-solution $u$ of the viscous stochastic Burgers' equation (\ref{eq0}) with $\nu > 0$. 
\end{theorem}

\begin{remark}
Theorems \ref{theorem-3} and \ref{theorem-4} can be extended in a straightforward manner to the full line $\mathbb R$ and to higher dimensions.
\end{remark}
\begin{remark}
We prove Theorem \ref{theorem-3} for a more general noise $\mathcal Q u \circ dW_t$, where $\mathcal Q$ is a first order linear differential operator, which includes the transport noise  as a special case.
For the sake of clarity, our proof deals only with one noise term $\mathcal{Q}u \circ dW_t$, however, we can readily extend this to cylindrical noise with countable set of first order linear differential operators
\[ \displaystyle\sum_{k=1}^{\infty}\mathcal{Q}_{k}(u) \circ \diff W^{k}_{t},\]
by imposing certain smoothness and boundedness conditions for the sum of the coefficients. We also prove Theorem \ref{theorem-4} for one noise term. 
\end{remark}

\subsection{Structure of the paper} This manuscript is organised as follows. In Section \ref{section2} we review some classical mathematical deterministic and stochastic background. We also fix the notations we will employ and state some definitions.
Section \ref{section3} contains the main results regarding shock formation in the stochastic Burgers' equation. Using a characteristic argument, we show that noise cannot prevent shocks from occurring for certain classes of $\{\xi_k(\cdot)\}_{k \in \mathbb N}$. Moreover, we prove that these shocks satisfy a Rankine-Hugoniot type condition in the weak formulation of the problem.
In Section \ref{section4}, we show local well-posedness of the stochastic Burgers' equation in Sobolev spaces and a blow-up criterion. We also establish global existence of smooth solutions of a viscous version of the stochastic Burgers' equation, which is achieved by proving a stochastic analogue of the maximum principle. In Section \ref{section5}, we provide conclusions, propose possible future research lines, and comment on several open problems that are left to study.

\section{Preliminaries and notation}     \label{section2}
Let us begin by reviewing some standard functional spaces and mathematical background that will be used throughout this article. Sobolev spaces are given by
\[ W^{s,p}:=\lbrace f\in L^{p}(\mathbb{T},\mathbb{R}):  (I-\partial_{xx})^{s/2}f\in L^{p}(\mathbb{T},\mathbb{R}) \rbrace, \]
for any $s\geq0$ and $p\in[1,\infty),$ equipped with the norm $||f||_{W^{s,p}}=||(I-\partial_{xx})^{s/2}f||_{L^{p}}$. We will also use the notation $\Lambda^{s}=(-\partial_{xx})^{s/2}$. Recall that $L^{2}$ based spaces are Hilbert spaces and may alternatively be denoted by $H^{s}=W^{s,2}$. For $s>0$, we also define $H^{-s}:=(H^{s})^{\star}$, i.e. the dual space of $H^s$. Let us gather here some well-known Sobolev embedding inequalities:
\begin{eqnarray}
\left\Vert f \right\Vert_{L^{4}} &\lesssim& \left\Vert f \right\Vert ^{1/2}_{L^{2}}  \left\Vert \partial_{x} f \right\Vert ^{1/2}_{L^{2}}, \label{Sob:ine1} \\
\left\Vert \partial_{x} f \right\Vert_{L^{4}} &\lesssim& \left\Vert f \right\Vert_\infty^{1/2} \left\Vert \partial_{xx} f \right\Vert ^{1/2}_{L^{2}}, \label{Sob:ine3} \\
\left\Vert f \right\Vert_\infty &\lesssim& \left\Vert f \right\Vert_{H^{1/2+\epsilon}}, \ \ \text{for every } \epsilon>0. \label{Sob:ine2}
\end{eqnarray}
Let us also recall the well-known commutator estimate of Kato and Ponce:
\begin{lemma}[\cite{katoponce}]
If $s\geq 0$ and $1<p<\infty$, then
\begin{equation}\label{katoponce}
\left|\left| \Lambda^{s}(fg)-f\Lambda^{s}(g)\right|\right|_{L^{p}} \leq C_{p}\left(\left\Vert \partial_x f \right\Vert_\infty||\Lambda^{s-1}g||_{L^{p}}+||\Lambda^{s}f||_{L^{p}} \left\Vert g \right\Vert_\infty \right).
\end{equation}
\end{lemma}
We will also use the following result as main tool for proving the existence results and blow-up criterion:
\begin{theorem}[\cite{stochBouss2018}]\label{generalcancellations}
Let $\mathcal{Q}$ be a linear differential operator of first order
\[ \mathcal{Q}= a(x) \partial_{x} + b(x) \]
where the coefficients are smooth and bounded. Then for $f\in H^{2}(\mathbb{T},\mathbb{R})$ we have 
\begin{equation}\label{eq:cancellation1:thm}
\langle \mathcal{Q}^2 f, f \rangle_{L^2} +  \langle \mathcal{Q} f, \mathcal{Q} f \rangle_{L^2} \lesssim ||f||_{L^2}^2.
\end{equation}
Moreover, if $f\in H^{2+s}(\mathbb{T},\mathbb{R})$, and $\mathcal{P}$ is a pseudodifferential operator of order $s,$ then
\begin{equation}\label{eq:cancellation2:thm}
\langle \mathcal{P} \mathcal{Q}^2 f, \mathcal{P} f \rangle_{L^2} +  \langle \mathcal{P} \mathcal{Q} f, \mathcal{P} \mathcal{Q} f \rangle_{L^2} \lesssim ||f||_{H^s}^2, 
\end{equation}
for every $s\in[1, \infty)$. 
\end{theorem}
\begin{remark}
Theorem \ref{generalcancellations} is fundamental for closing the energy estimates when showing well-posedness of the stochastic Burgers' equation. It permits reducing the order of a sum of terms which in principle seems hopelessly singular.
\end{remark}

\begin{remark}
As pointed out in the previous observation, in order to extend Theorem \ref{theorem-3} to cylindrical noise of the form 
$\sum_{k=1}^\infty \xi_k(x) \partial_x u(t,x) \circ \diff W_t^k$ (or in general $\sum_{k=1}^{\infty}\mathcal{Q}_{k}(u)\circ \diff W^{k}_{t}$), it is fundamental to show the cancellation property provided by Theorem \ref{generalcancellations} for such noises. This can be done under some mild Sobolev regularity assumption on the coefficients $\xi_{k}$ (respectively $a_{k}, b_{k}$). In particular, one has to precisely compute the constants $C_{k}$ hidden on the right hand-side of \eqref{eq:cancellation1:thm} and \eqref{eq:cancellation2:thm}, whose sum a priori does not have to converge. We refer the reader to Lemma A.5 in \cite{AlonHaoRohde} for a detailed calculation to deal with this extension.
\end{remark}
Next, we briefly recall some aspects of the theory of stochastic analysis. Fix a stochastic basis
 $\mathcal{S}=(\Xi,\mathcal{F},\lbrace\mathcal{F}_{t}\rbrace_{t\geq 0}, \mathbb{P},\lbrace W^{k}\rbrace_{k\in\mathbb{N}}),$ that is, a filtered probability space together with a sequence $\lbrace W^{k} \rbrace_{k\in\mathbb{N}}$ of scalar independent Brownian motions relative to the filtration $\lbrace\mathcal{F}_{t}\rbrace_{t\geq0}$ satisfying the usual conditions.  \\ \\
Given a stochastic process $X\in L^{2}(\Xi;L^{2}([0,\infty);L^{2}(\mathbb{T},\mathbb{R}))),$ the Burkholder-Davis-Gundy inequality is given by
\begin{equation}\label{BGD:ineq}
 \mathbb{E}\left[ \displaystyle \sup_{s \in [0,T]}\left| \int_{0}^{t} X_{s} \diff W_ {s}\right|^{p} \right] \leq C_p \mathbb{E} \left[ \int_{0}^{T} |X_s|^{2} \diff t \right]^{p/2},
\end{equation}
for any $p\geq 1$ and $C_{p}$ an absolute constant depending on $p$.\\ \\
We also state the celebrated It\^o-Wentzell formula, which we use throughout this work.
\begin{theorem}[\cite{kunita81}, Theorem 1.2] \label{Ito-Wentzell}
For $0\leq t<\tau$, let $u(t,\cdot)$ be $C^3$ almost surely, and $u(\cdot,x)$ be a continuous semimartingale satisfying the SPDE
\begin{align}
u(t,x) = u(0,x) + \sum^\infty_{j=0} \int_{0}^{t} \sigma_j(s,x) \circ \diff N_s^j,
\end{align}
where $\{N_t^j\}_{j = 0}^\infty$ is a family of continuous semimartingales and $\{\sigma_j(t,x)\}_{j=0}^\infty$ is also a family of continuous semimartingales that are $C^2$ in space for $0\leq t<\tau$. Also, let $X_t$ be a continuous semimartingale. Then, we have the following
\begin{align} \label{Ito-Wentzell-formula}
u(t,X_t) = u(0,X_0) + \sum^\infty_{j=0} \int^t_0 \sigma_j(s,X_s) \circ \diff N^j_s + \int^t_0 \partial_x u (s,X_s) \circ \diff X_s.
\end{align}
\end{theorem}

Let us also introduce three different notions of solutions:
\begin{definition}[Local solution] \label{localsol}
An $H^s$-local solution $(\tau,u)$ of \eqref{eq0}, $s \geq 2,$ is a stochastic process $u:\Xi \times [0,\tau] \times \mathbb{T} \rightarrow \mathbb{R},$ where $\tau: \Xi \rightarrow [0, \infty)$ is a stopping time such that $u_{t \wedge \tau}$ is adapted to $(\mathcal{F}_{t})_t$ and we have:
\begin{itemize}
    \item a.s. $u$ has paths of class $C([0,\tau]; H^{s}(\mathbb{T})).$
    \item It holds
    {\small \begin{align*}
        u_{\tau'} - u_0  + \int_0^{\tau'} u_s\partial_{x}u_s \diff s + \sum_{i=1}^{N} \int_0^{\tau'} \xi_{i} \partial_{x}u_s  \diff W^{i}_{s}= \dfrac{1}{2} \sum_{i=1}^{N} \int_{0}^{\tau'} (\xi_{i} \partial_{x})^{2}u_s \diff s,
    \end{align*}}
    as an identity in $L^2(\mathbb{T})$, a.s., for bounded stopping times $\tau' \leq \tau$. 
\end{itemize}
When the stopping time is clear from the context, we simply write that $u$ is a solution.
\end{definition}

\begin{definition}[Maximal solution]\label{def:maximal}
A maximal solution $(\tau_{max},u)$ of \eqref{eq0} is a stochastic process $u:\Xi \times [0,\tau_{max})\times \mathbb{T} \rightarrow \mathbb{R},$ for a stopping time $\tau_{max}: \Xi \to [0,\infty]$ \footnote{Unlike a stopping time for a local solution, a stopping time for a maximal solution is allowed to take the value infinity.}, satisfying the following conditions: 
\begin{itemize}
\item $\mathbb{P} (\tau_{max} >0) = 1, $ where $\tau_{max} = \lim_{n \rightarrow \infty} \tau_n$ for an increasing sequence of stopping times $\{\tau_n\}_{n=1}^\infty$.
\item $(\tau_{n},u)$ is a local solution in the sense of Definition \ref{localsol}, for all $n\in \mathbb{N}$.
\item If $(\tau',u')$ is another pair satisfying the above two conditions and $u'=u$ on $[0,\tau'\wedge \tau_{max} )$, then $\tau'\leq  \tau_{max}$, a.s.
\end{itemize}
A maximal solution is said to be global if $\tau_{max}=\infty$, a.s.
\end{definition}

\begin{definition}[Weak solution] \label{weak-sol}
A (spatially) weak solution $(\tau,u)$ of \eqref{eq0} is a stochastic process $u:\Xi \times [0,\tau] \times \mathbb{T} \rightarrow \mathbb{R},$ where $\tau: \Xi \rightarrow [0, \infty)$ is a stopping time, such that for any test function $\phi \in C^\infty(\mathbb T)$, $\langle u_{t \wedge \tau}, \phi \rangle_{L^2}$ is adapted to $(\mathcal{F}_{t})_t$ and we have:
\begin{itemize}
\item
a.s. $u$ has paths of class $C([0,\tau]; L^2(\mathbb{T})).$
\item It holds
\begin{align}
&  \langle u_{\tau'}, \phi \rangle_{L^2} = \langle u_0, \phi \rangle_{L^2} +   \frac12 \int_0^{\tau'} \langle (u_s)^2, \partial_x \phi \rangle_{L^2} \diff s \nonumber \\
& \hspace{10pt} + \sum_{i=1}^N \int_0^{\tau'} \langle u_s, \partial_x  (\xi_i \phi ) \rangle_{L^2} \diff W_s^i + \dfrac{1}{2} \sum_{i=1}^{N} \int_{0}^{\tau'} \langle u_s, \partial_x(\xi_i \partial_x(\xi_i \phi)) \rangle_{L^2} \diff s , \label{stoch-burger-weak}
\end{align}
a.s., for bounded stopping times $\tau' \leq \tau$.
\end{itemize}
\end{definition}

\subsubsection*{Notations:} Let us stress some notations that we will use throughout this work. We will denote the Sobolev $L^{2}$ based spaces by $H^{s}(\text{domain}, \text{target space})$. However, we will sometimes omit the domain and target space and just write $H^s$, when these are clear from the context. $a\lesssim b$ means there exists $C$ such that $a \leq Cb$, where $C$ is a positive universal constant that may depend on fixed parameters and constant quantities. Note also that this constant might differ from line to line. It is also important to remind that the condition ``almost surely'' is not always indicated, since in some cases it is obvious from the context.

\section{Shocks in Burgers' equation with stochastic transport}    \label{section3}
 Recall that we are dealing with a stochastic Burgers' equation of the form
\begin{align*}
\diff u + \left(u(t,x) \diff t + \sum_{k=1}^\infty \xi_k(x) \circ \diff W_t^k \right) \cdot \partial_x u = \nu \partial_{xx} u \diff t,
\end{align*}
for $x \in \mathbb T$ or $\mathbb R,$ where $\nu \geq 0$ is constant, $\{\xi_k(x)\}_{k \in \mathbb N}$ is an orthonormal basis of some separable Hilbert space $\mathcal H,$ and $\circ$ means that the integration is carried out in the Stratonovich sense.
In this section, we study the problem of whether shocks can form in the inviscid Burgers' equation with stochastic transport. By using a characteristic argument, we prove that for some classes $\{\xi_k(x)\}_{k \in \mathbb N}$, the transport noise cannot prevent shock formation. We also consider a weak formulation of the problem and prove that the shocks satisfy a stochastic version of the Rankine-Hugoniot condition.

\subsection{Inviscid Burgers' equation with stochastic transport}
The inviscid Burgers' equation with stochastic transport is given by
\begin{align} \label{1d-stoch-burger}
\diff u + \left(u(t,x) \,\diff t + \sum_{k=1}^\infty \xi_k(x) \circ \diff W_t^k\right) \cdot \partial_x u = 0,
\end{align}
which in integral form is interpreted as
\begin{align} \label{1d-stoch-burger-int}
u(t,x) = u(0,x) - \int_0^t \left(u(s,x) \partial_x u(s,x) \,\diff s + \sum_{k=1}^\infty \xi_k(x) \partial_x u(s,x) \circ \diff W_s^k \right),
\end{align}
for all $x \in \mathbb T$ or $\mathbb R$. Also, we will assume throughout this paper that the initial condition is positive, that is, $u(0,x) > 0$ for all $x \in \mathbb T$ or $\mathbb R$.

Consider a process $X_t$ that satisfies the {\em 
Stratonovich} SDE 
\begin{align}
\diff X_t &= u(t,X_t) \diff t + \sum_{k=1}^\infty \xi_k(X_t) \circ \diff W_t^k, \label{characteristic-sde}
\end{align}
which in It\^o form, reads
\begin{align}
\diff X_t = \left(u(t,X_t) + \frac12 \sum_{k=1}^\infty \xi_k(X_t) \partial_x\xi_k(X_t) \right)\diff t + \sum_{k=1}^\infty \xi_k(X_t) \diff W_t^k. \label{char-ito}
\end{align}
We call this process the {\em characteristic} of $\eqref{1d-stoch-burger}$, analogous to the characteristic lines in the deterministic Burgers' equation. We assume the following conditions on $\{\xi_k(\cdot)\}_{k \in \mathbb N}$.

\begin{assump} \label{xi-assump}
$\xi_k$ is smooth for all $k \in \mathbb N$ and together with the Stratonovich-to-It\^o correction term $\varphi(x):=\frac12 \sum_{k=1}^\infty \xi_k(x) \partial_x\xi_k(x)$, satisfy the following:
\begin{itemize}
    \item Lipschitz continuity
\begin{align}
|\varphi(x) - \varphi(y)| \leq C_0 |x-y|, \quad 
|\xi_k(x) - \xi_k(y)| \leq C_k |x-y|, \quad k \in \mathbb N, \label{lipschitz}
\end{align}
\item Linear growth condition
\begin{align}
|\varphi(x)| \leq D_0 (1 + |x|), \quad |\xi_k(x)| \leq D_k (1+|x|), \quad k \in \mathbb N \label{linear-bound}
\end{align}
\end{itemize}
for real constants $C_0, C_1, C_2, \ldots$ and $D_0, D_1, D_2, \ldots$ with
\begin{align}
    \sum_{k=1}^\infty C_k^2 < \infty, \quad \sum_{k=1}^\infty D_k^2 < \infty. \label{summability}
\end{align}
\end{assump}
Provided $u(t,\cdot)$ is sufficiently smooth and bounded (hence satisfying Lipschitz continuity and linear growth) until some stopping time $\tau$, and $\{\xi_k(\cdot)\}_{k \in \mathbb N}$ satisfies the conditions in Assumption \ref{xi-assump}, the characteristic equation \eqref{char-ito} is locally well-posed.
One feature of the multiplicative noise in \eqref{1d-stoch-burger} is that $u$ is transported along the characteristics, that is, we can show that $u(t,x) = (\Phi_t)_*u_0(x)$ for $0\leq t < \tau_{max},$ where $\Phi_t$ is the stochastic flow of the SDE \eqref{char-ito}, $(\Phi_t)_*$ represents the pushforward by $\Phi_t,$ and $(\tau_{max}, X_t)$ is the maximal solution of \eqref{char-ito}. This is an easy corollary of the It{\^o}-Wentzell formula \eqref{Ito-Wentzell-formula}.

\begin{corollary} \label{advection}
Let $u(t,\cdot)$ be $C^3 \cap L^\infty$ in space for $0<t<\tau$. Assume also that $u(\cdot,x)$ is a continuous semimartingale satisfying \eqref{1d-stoch-burger-int}, $\partial_x u (\cdot,x)$ is a continuous semimartingale satisfying the spatial derivative of \eqref{1d-stoch-burger-int}, and $\{\xi_k(\cdot)\}_{k \in \mathbb N}$  satisfies the conditions in Assumption \ref{xi-assump}. If $(\tau_{max},X_t)$ is a maximal solution to \eqref{char-ito}, then $u(t,X_t) = u(0,X_0)$ almost surely for $0<t<\tau_{max}$.
\end{corollary}

\begin{remark}
Notice that due to our local well-posedness result (Theorem \ref{theorem-3}) and the maximum principle (Proposition \ref{max-principle}), one has $u_t \in C^3 \cap L^\infty$ for $t < \tau_{max}$ provided $u_0$ is smooth enough and bounded. For instance, $u_0 \in H^4 \cap L^\infty$ is sufficient.
\end{remark}

\begin{proof}[Proof of Corollary \ref{advection}]
Note that under the given assumptions, $\sigma_0(t,x) := u(t,x)\partial_x u (t,x),$ and $\sigma_k(t,x) := \xi_k(x)\partial_x u(t,x)$ for all $k \in \mathbb N,$ satisfy the conditions in Theorem \ref{Ito-Wentzell}. We take $N_t^0 = t$ and $N_t^k = W_t^k$ for $k \in \mathbb N$. Using the It{\^o}-Wentzell formula \eqref{Ito-Wentzell-formula} for the stochastic field $u(t,x)$ satisfying \eqref{1d-stoch-burger-int}, and the semimartingale $X_t$, we obtain
\begin{align}
u(t,X_t) &= u(0,X_0) - \int_0^t \left(u(s,X_s)\partial_x u (s,X_s) \diff s + \sum_{k=1}^\infty \xi_k(X_s)\partial_x u (s,X_s) \circ \diff W_s^k \right) \nonumber \\
&+ \int_0^t \partial_x u (s,X_s) \circ \diff X_s \nonumber \\
&= u(0,X_0) - I_1 + I_2. \nonumber
\end{align}

Now, we have $I_1 = I_2$ almost surely so indeed, $u(t,X_t) = u(0,X_0)$ almost surely for $0<t<\tau_{max}$.
\end{proof}

\subsection{Results on shock formation}
In order to investigate the crossing of characteristics in the stochastic Burgers' equation \eqref{1d-stoch-burger} with transport noise, we define the {\em first crossing time} $\tau$ as
\begin{align} \label{first-crossing}
\tau := \inf_{\substack{a,b \in \mathbb R \\
a \neq b }}
\left\{ \inf \left\{t>0 : X_t^a = X_t^b \right\} \right\},
\end{align}
where $X_t^a, X_t^b$ are two characteristics that solve the SDE \eqref{characteristic-sde} with initial conditions $X_0^a = a$ and $X_0^b = b$. This gives us the first time when two characteristics intersect.
In the following, we will show that in the special case $\xi_1(x) = \alpha x + \beta$ (where we only consider one noise term and the other terms $\xi_k$ are identically zero for $k = 2,3,\ldots$), the first crossing time is equivalent to the first hitting time of the integrated geometric Brownian motion.
We note that in this case, equation \eqref{char-ito} is explicitly solvable, where the general solution is given by
\begin{align} \label{characteristic-sol}
X_t^\gamma &= e^{\alpha W_t} \left(\gamma + \left(u_0(\gamma) - \alpha \beta \right) \int^t_0 e^{-\alpha W_s} \diff s + \beta \int^t_0 e^{-\alpha W_s} \diff W_s \right).
\end{align}

\begin{proposition} \label{cross-time-law}
The first crossing time of the inviscid stochastic Burgers' equation \eqref{1d-stoch-burger} with $\xi_1(x) = \alpha x + \beta$ for constants $\alpha, \beta \in \mathbb R$ and $\xi_k(\cdot) \equiv 0$ for $k = 2,3,\ldots$ is equivalent to the first hitting time for the integrated geometric Brownian motion $I_t := \int^t_0 e^{-\alpha W_s}ds$.
\end{proposition}
\begin{proof}
Consider two arbitrary characteristics $X_t^a$ and $X_t^b$ with $X_0^a = a$ and $X_0^b = b$. From \eqref{characteristic-sol}, one can check that $X_t^a = X_t^b$ if and only if
\begin{align*}
I_t := \int^t_0 e^{-\alpha W_s} \diff s = -\frac{b-a}{u_0(b)-u_0(a)}.
\end{align*}

Now, since the left-hand side is continuous, strictly increasing with $I_0 = 0,$ and independent of $a$ and $b$, we have
\begin{align}
\tau &= \inf_{\substack{a,b \in \mathbb R \\
a \neq b }} \left\{ \inf \left\{t>0 : I_t = - \frac{b-a}{u_0(b)-u_0(a)} \right\} \right\} \nonumber \\
&= \begin{cases}
\inf \left\{t>0 : I_t = \theta(u_0)^{-1} \right\}, \quad & \text{ if } \theta(u_0) > 0 \\
\infty, \quad & \text{ if } \theta(u_0) = 0
\end{cases},
\label{hitting-time}
\end{align}
where
\begin{align*}
\theta(u_0) := \sup_{\substack{a,b \in \mathbb R \\
a \neq b }} \left\{ \zeta(a,b) \right\},
\quad \zeta(a,b) = \begin{cases}
\frac{|u_0(a)-u_0(b)|}{|a - b|}, \quad &\text{if } \frac{u_0(a)-u_0(b)}{a - b} < 0 \\
0, \quad &\text{otherwise}
\end{cases},
\end{align*}
is the steepest negative slope of $u_0$. Hence, the first crossing time is equivalent to the first hitting time of the process $I_t$.
\end{proof}

\begin{remark}
Note that the constant $\beta$ does not affect the first crossing time, hence we can set $\beta=0$ without loss of generality. Also in the following, we simply write $\xi(\cdot)$ without the index when we only consider one noise term.
\end{remark}

As an immediate consequence of Proposition \ref{cross-time-law}, we prove that the transport noise with $\xi(x) = \alpha x$ {\em cannot} prevent shocks from forming almost surely in the stochastic Burgers' equation \eqref{1d-stoch-burger}.

\begin{corollary}
Let $\xi(x) = \alpha x$ for some $\alpha \in \mathbb R$. If the initial profile $u_0$ has a negative slope, then $\tau < \infty$ almost surely.
\end{corollary}

\begin{proof}
To prove this, it is enough to show that
\begin{align*}
\lim_{t \rightarrow \infty} \int^t_0 e^{\alpha W_s} \diff s = \infty \quad a.s.
\end{align*}
where we have assumed $\alpha > 0,$ without loss of generality, and $W_{\bullet} : \mathbb R_{\geq 0} \times \Xi \rightarrow \mathbb R$ is the standard Wiener process on the Wiener space $(\Xi, \mathcal F, \mathbb P),$ adapted to the natural filtration $\mathcal F_t$. This implies that $\tau < \infty$ a.s. by Proposition \ref{cross-time-law}.

First, define the set
\begin{align*}
A = \left\{\omega \in \Xi : \lim_{t \rightarrow \infty} \int^t_0 e^{\alpha W_s(\omega)} \diff s < \infty \right\} \subset \Xi.
\end{align*}

Fixing $\omega \in A$, choose $t_1, t_2, \ldots \in \mathbb R_{\geq 0}$  with $t_n < t_{n+1},$ such that $\lim_{n \rightarrow \infty} t_n = \infty$ and $\liminf_{n \rightarrow \infty} (t_{n+1}-t_n) > 0$, and consider the sequence
\begin{align*}
I_n(\omega) = \int^{t_n}_0 e^{\alpha W_s(\omega)} \diff s, \quad n=1, 2, \ldots
\end{align*}

Clearly, $\{I_n(\omega)\}_{n \in \mathbb N}$ is monotonic increasing, and it is also bounded since $\omega \in A$. Hence, it is convergent by the monotone convergence theorem, and in particular, it is a Cauchy sequence. Therefore we have
\begin{align*}
\lim_{n \rightarrow \infty} |I_{n+1}(\omega) - I_n(\omega)| = \lim_{n \rightarrow \infty} \int^{t_{n+1}}_{t_n} e^{\alpha W_s(\omega)}\diff s = 0.
\end{align*}

Since the integrand is strictly positive, this implies $\lim_{t \rightarrow \infty} e^{\alpha W_t(\omega)} = 0,$ and hence $W_t(\omega) \rightarrow -\infty$. On the other hand, for $\omega \in \Xi$ such that $W_t(\omega) \rightarrow -\infty$, it is easy to see that $\omega \in A$.
This implies that under the identification $\Xi \cong C([0,\infty);\mathbb R)$, the set $A$ is equivalent to the set of Wiener processes $W_t$ with $W_t \rightarrow -\infty$, which is open in $C([0,\infty); \mathbb R)$ endowed with the norm $\|\cdot \|_{\infty}$ and therefore measurable.
In particular, for $\omega \in A$, we have
\begin{align*}
\limsup_{t \rightarrow \infty} W_t(\omega) = -\infty,
\end{align*}
but since $\limsup_{t \rightarrow \infty} W_t = +\infty,$ a.s., this implies $\mathbb P(A) = 0$.
\end{proof}

In the following, we show that for a broader class of $\{\xi_k(\cdot)\}_{k \in \mathbb N}$, shock formation occurs in expectation provided the initial profile has a sufficiently negative slope. Moreover, no new shocks can develop from positive slopes.
We show this by looking at how the slope $\partial_x u$ evolves along the characteristics $X_t$, which resembles the argument given in \cite{stochasticCH} for the stochastic Camassa-Holm equation.

\begin{theorem}
Consider a characteristic $X_t,$ and a smooth initial profile $u(0,x)=u_0(x)$ such that $\partial_x u(0,X_0) = - \sigma < 0$. If
\[
\frac12\sum_{k=1}^\infty \left( (\partial_x \xi_k(x))^2 - \xi_k(x) \partial_{xx} \xi_k(x)\right) > -\sigma, \quad \forall x \in \mathbb R,
\]
then there exists $0<t_*<\infty$ such that $\lim_{t \rightarrow t_*} \mathbb E [\partial_x u(t,X_t)] = - \infty$.
On the other hand, if $\partial_x u(0,X_0) \geq 0$ and
\[
\frac12\sum_{k=1}^\infty \left( (\partial_x \xi_k(x))^2 - \xi_k(x) \partial_{xx} \xi_k(x)\right) < \infty, \quad \forall x \in \mathbb R,
\]
then $\partial_x u(t,X_t) < \infty$ almost surely  for all $t > 0$.
\end{theorem}

\begin{proof}
Taking the spatial derivative of \eqref{1d-stoch-burger-int}, and evaluating the stochastic field $\partial_x u(t,x)$ along the semimartingale $X_t$ by the Ito-Wentzell formula \eqref{Ito-Wentzell-formula} (again, this is valid due to the local well-posedness result, Theorem \ref{theorem-3}), the process $Y_t := \partial_x u(t,X_t)$ together with $X_t$ satisfy the following coupled Stratonovich SDEs
\begin{align}
\diff X_t &= u(t,X_t)\diff t + \sum_{k=1}^\infty \xi_k(X_t) \circ \diff W_t^k \label{x-eq}, \\
\diff Y_t &= -Y_t^2\diff t - \sum_{k=1}^\infty \partial_x \xi_k (X_t) Y_t \circ \diff W_t^k . \label{y-eq}
\end{align}
In It\^o form, this reads
\begin{align}
\diff X_t &= \left(u(t,X_t) + \frac12 \sum_{k=1}^\infty \xi_k(X_t) \partial_x \xi_k(X_t)\right) \diff t + \sum_{k=1}^\infty \xi_k(X_t) \diff W_t^k, \label{x-eq-ito} \\
\diff Y_t &= \left(-Y_t^2 + \frac12 Y_t \sum_{k=1}^\infty \left( (\partial_x \xi_k(x))^2 - \xi_k(x) \partial_{xx} \xi_k(x)\right) \right) \diff t - \sum_{k=1}^\infty \partial_x \xi_k(X_t) Y_t \diff W_t^k .\label{y-eq-ito}
\end{align}

Taking the expectation of $\eqref{y-eq-ito}$ on both sides, we obtain
\begin{align}
\frac{\diff \mathbb E[Y_t]}{\diff t} = -\mathbb E[Y_t^2] + \frac12 \mathbb E\left[Y_t\sum_{k=1}^\infty \left( (\partial_x \xi_k(x))^2 - \xi_k(x) \partial_{xx} \xi_k(x)\right)\right]. \label{exp-y}
\end{align}

Now, assume that there exists a constant $C \in \mathbb R$ such that
\begin{align}\label{positivebound}
    C \leq \sum_{k=1}^\infty \left( (\partial_x \xi_k(x))^2 - \xi_k(x) \partial_{xx} \xi_k(x)\right),
\end{align} for all $x \in \mathbb R$.
If $Y_0 = -\sigma < 0$, we have $Y_t < 0$ for all $t > 0,$ since $Y = 0$ is a fixed line in the phase space $(X,Y)$ and therefore cannot be crossed. Hence from \eqref{positivebound}, we have
\begin{align*}
\mathbb E\left[Y_t\sum_{k=1}^\infty \left( (\partial_x \xi_k(X_t))^2 - \xi_k(X_t) \partial_{xx} \xi_k(X_t)\right)\right] \leq C \mathbb E[Y_t],
\end{align*}
and \eqref{exp-y} becomes,
\begin{align*}
\frac{\diff \mathbb E[Y_t]}{\diff t} &\leq -\mathbb E[Y_t^2] + \frac{C}{2} \mathbb E\left[Y_t\right] \\
&= -(\mathbb E[Y_t^2] - \mathbb E[Y_t]^2) - \mathbb E[Y_t]^2 + \frac{C}{2} \mathbb E\left[Y_t\right] \\
& \leq - \mathbb E[Y_t]^2 + \frac{C}{2} \mathbb E\left[Y_t\right], 
\end{align*}
since $\mathbb E[Y_t^2] - \mathbb E[Y_t]^2 = \mathbb E\left[(Y_t - \mathbb E[Y_t])^2\right] \geq 0$.\\\\
Solving this differential inequality, we get
\begin{align*}
    \mathbb E[Y_t] \leq \begin{cases}
    \frac{-\sigma  e^{Ct/2}}{1 - \frac{2\sigma}{C} \left(e^{Ct/2} - 1\right)}, \quad &\text{if } C \neq 0 \\
    \frac{1}{t - \frac{1}{\sigma}}, \quad &\text{if } C=0
    \end{cases}.
\end{align*}
The right-hand side tends to $-\infty$ in finite time provided $-\sigma < C/2$.

Hence, if 
\begin{align*}
    -\sigma < C/2 \leq \frac12 \sum_{k=1}^\infty \left( (\partial_x \xi_k(x))^2 - \xi_k(x) \partial_{xx} \xi_k(x)\right),
\end{align*} 
for all $x \in \mathbb R$, then there exists $t_* < \infty$ such that $\lim_{t \rightarrow t_*} \mathbb E [u_x(t, X_t)] = - \infty$.

Similarly, if $
\frac12\sum_{k=1}^\infty \left( (\partial_x \xi_k(x))^2 - \xi_k(x) \partial_{xx} \xi_k(x)\right) < D$ for some $D \in \mathbb R$, then for $Y_0 > 0$ we have again
\begin{align*}
    \frac{\diff \mathbb E[Y_t]}{\diff t} \leq -\mathbb E[Y_t]^2 + D \mathbb E[Y_t].
\end{align*}
One can check that $\mathbb E[Y_t] < \infty$ for all $t>0,$ which implies $Y_t < \infty$ almost surely.
\end{proof}

\begin{remark}
Blow-up in expectation does not imply pathwise blow-up. It is merely a necessary condition, which suggests that the law of $\partial_x u$ becomes increasingly fat-tailed with time, making it more likely for it to take extreme values. Nonetheless, it is a good indication of blow-up occurring with some probability.
 \end{remark}

\begin{example}
Consider the set $\{\xi_k(x)\}_{k \in \mathbb N} = \left\{\frac{1}{k^2}\sin(kx), \frac{1}{k^2}\cos(kx)\right\}_{k \in \mathbb N}$, which forms an orthogonal basis for $L^2(\mathbb T)$. Then, one can easily check that
\begin{align*}
    0 < \sum_{k=1}^\infty \left( (\partial_x \xi_k(x))^2 - \xi_k(x) \partial_{xx} \xi_k(x)\right) < \infty,
\end{align*}
for all $x \in \mathbb T,$ so blow-up occurs in expectation for any initial profile with negative slope, but no new shocks can form from positive slopes.
\end{example}

\subsection{Weak solutions}
We saw that if the initial profile $u_0$ has a negative slope, then shocks may form in finite time (almost surely in the linear case $\xi(x) = \alpha x$), so solutions to \eqref{1d-stoch-burger} cannot exist in the classical sense. This motivates us to consider {\em weak solutions} to \eqref{1d-stoch-burger} in the sense of Definition \ref{weak-sol}.

Suppose that the profile $u$ is differentiable everywhere except for a discontinuity along the curve $\gamma = \left\{(t,s(t)) \in [0,\infty) \times M \right\}$, where $M = \mathbb T$ or $\mathbb R$. Then the curve of discontinuity must satisfy the following for $u$ to be a solution of the integral equation \eqref{stoch-burger-weak}.

\begin{proposition}[Stochastic Rankine-Hugoniot condition] \label{stoch-RH-condition}
The curve of discontinuity $s(t)$ of the stochastic Burgers' equation in weak form \eqref{stoch-burger-weak} satisfies the following SDE
\begin{align}
\diff s_t = \frac12 \left[(u_-(t,s(t)) + u_+(t,s(t))\right]\,\diff t + \sum_{k=1}^\infty \xi_k(s(t)) \circ \diff W_t^k, \label{stoch-RH}
\end{align}
where $u_\pm (t,s(t)) := \lim_{x \rightarrow s(t)^\pm} u(t,x)$ are the left and right limits of $u$.
\end{proposition}

The main obstacle here is that the curve $s(t)$ is not piecewise smooth and therefore we cannot apply the standard divergence theorem, which is how the Rankine-Hugoniot condition is usually derived. Extending classical calculus identities such as Green's theorem on domains with non-smooth boundaries is a tricky issue, but fortunately, there have been several works that extend this result to non-smooth but rectifiable boundaries in \cite{Shapiro1957}, and to non-rectifiable boundaries in \cite{harrison1992,harrison1993,harrison1999,lyons2006}.

\begin{lemma}[Green's theorem for non-smooth boundaries] \label{shapiro}
Let $\Omega$ be a bounded domain in the $(x,y)$-plane such that its boundary $\partial \Omega$ is a Jordan curve and let $u,v$  be sufficiently regular functions in $\Omega$ (see remark \ref{remark-green} below). Then
\begin{align}
\int_\Omega \text{div}(u,v) \diff x \diff y = \oint_{\partial \Omega} \left(u \diff y - v \diff x \right),
\end{align}
where the contour integral on the right-hand side can be understood as a limit of a standard contour integral along a smooth approximation of the boundary. Here, the integral is taken in the anti-clockwise direction of the contour.
\end{lemma}
\begin{remark} \label{remark-green}
For the above to hold, there must be a pay-off between the regularity of $\partial \Omega$ and the functions $u,v$ (i.e. the less regular the boundary, the more regular the integrand). In particular, the following condition is known:
\begin{itemize}
    \item $\partial \Omega$ has box-counting dimension $d < 2$ and $u,v$ is $\alpha$-H\"older continuous for any $\alpha > d-1$ (Harrison and Norton \cite{harrison1992}).
\end{itemize}
\end{remark}

\begin{proof}[Proof of Theorem \ref{stoch-RH-condition}]
We provide a proof in the case $M = \mathbb T$ with only one noise term. Extending it to the case $M=\mathbb R$ and countably many noise terms is straightforward. Take the atlas $\{(U_1,\varphi_1),(U_2,\varphi_2)\}$ on $\mathbb T =\mathbb R/\mathbb Z,$ where $U_1 := (0,1)$, $\varphi_1:(0,1) \rightarrow U_1$ and $U_2 := (-\frac12,\frac12)$, $\varphi_2 : (0,1) \rightarrow U_2$.
Without loss of generality, assume that the shock $s(\cdot)$ starts at time $t=0$.

Now, consider a sequence $0 = \tau_0 < \tau_1 < \tau_2 < \ldots,$ with $\lim_{n \rightarrow \infty} \tau_n = \infty$ such that for all $n \in\mathbb{N}$, the curve $\gamma_n := \{s(t) : t \in [\tau_{n-1}, \tau_n)\}$ is contained in either one of the charts $U_1$ or $U_2$. For convenience, we denote by $(U_n, \varphi_n)$ to mean the chart $(U_1,\varphi_1)$ or $(U_2,\varphi_2)$ that contains $\gamma_n$.
In local coordinates, we split the domain $\Omega_n := [\tau_{n-1},\tau_n) \times \varphi_n^{-1}(U_n)$ into two regions (see figure \ref{RH-figure})
\begin{align}
&\Omega_-^n :=  \left\{(t, x) \in \mathbb [\tau_{n-1},\tau_n) \times (0,1) : x \in (0, s(t)) \right\}, \\
&\Omega_+^n :=  \left\{(t, x) \in \mathbb [\tau_{n-1},\tau_n) \times (0,1) : x \in (s(t), 1) \right\}.
\end{align}
For $n \in \mathbb N$, consider the following integrals
\begin{align}
I_n &= \iint_{\Omega_-^n} \left( \left(u \partial_t \varphi + \frac12 u^2 \partial_x \varphi \right) \diff t + u \partial_x \left(\varphi(t,x) \xi(x) \right) \circ \diff W_t \right) \diff x \nonumber \\
& =\iint_{\Omega_-^n} \text{div}_{x,t} \left(\frac12 \varphi u(t,x)^2, \, \varphi(t,x) u(t,x) \right) \diff x \diff t + \int_{\tau_{n-1}}^{\tau_n}\left(\int_0^{s(t)} \partial_x \left(\varphi(t,x) \xi(x) u(t,x) \right) \diff x\right) \circ \diff W_t \nonumber \\
& - \underbrace{\iint_{\Omega_-^n} \varphi \left(\diff u + u\partial_x u \diff t + \xi(x) \partial_x u \circ \diff{W_t} \right) \diff x}_{=0}, \text{ and}\nonumber
\end{align}
\begin{align}
J_n &= \iint_{\Omega_+^n} \left( \left(u \partial_t \varphi + \frac12 u^2 \partial_x \varphi \right) \diff t + u \partial_x \left(\varphi(t,x) \xi(x) \right) \circ \diff W_t \right) \diff x \nonumber \\
&= \iint_{\Omega_+^n} \text{div}_{x,t} \left(\frac12 \varphi u(t,x)^2, \, \varphi(x,t) u(t,x) \right) \diff x \diff t + \int_{\tau_{n-1}}^{\tau_n}\left(\int^1_{s(t)} \partial_x \left(\varphi(t,x) \xi(x) u(t,x) \right) \diff x\right) \circ \diff W_t \nonumber \\
& - \underbrace{\iint_{\Omega_+^n} \varphi \left(\diff u + u\partial_x u \diff t + \xi(x) \partial_x u \circ \diff{W_t} \right) \diff x}_{=0}. \nonumber
\end{align}

Then by Lemma \ref{shapiro}, we have
\begin{align}
I_n &= \oint_{\partial \Omega_-^n} \left(\frac12 \varphi u(t,x)^2 \diff t - \varphi(t,x) u(t,x) \diff x \right) + \int^{\tau_n}_{\tau_{n-1}} \varphi(t,s(t)) \xi(s(t)) u_-(t,s(t)) \circ \diff W_t \nonumber \\
&= -\int^{\tau_n}_{\tau_{n-1}} \varphi(t,s(t)) \left(u_-(t,s(t)) \diff s_t - \frac12 u_-(t,s(t))^2 \diff t - \xi(s(t)) u_-(t,s(t)) \circ \diff W_t\right) \nonumber \\
&+ \left(\int_A-\int_B-\int_C\right) \left(\frac12 \varphi u(t,x)^2 \diff t - \varphi(t,x) u(t,x) \diff x \right), \nonumber
\end{align}
where
\begin{align*}
A:=\{(\tau_{n-1},x) : x \in (0,s(\tau_{n-1}))\},\, B:=\{(\tau_{n},x) : x \in (0,s(\tau_{n}))\}, \, C:=\{(t,0) : t \in (\tau_{n-1},\tau_n)\},
\end{align*}
and
\begin{align}
J_n &= \oint_{\partial \Omega_+^n} \left(\frac12 \varphi u(t,x)^2 \diff t - \varphi(x,t) u(t,x) \diff x \right) - \int^{\tau_n}_{\tau_{n-1}} \varphi(t,s(t)) \xi(s(t)) u_+(t,s(t)) \circ \diff W_t \nonumber \\
&= \int_{\tau_{n-1}}^{\tau_n} \varphi(t,s(t)) \left(u_+(t,s(t)) \diff s_t - \frac12 u_+(t,s(t))^2 \diff t - \xi(s(t)) u_+(t,s(t)) \circ \diff W_t\right) \nonumber \\
&+ \left(\int_D+\int_E+\int_F\right) \left(\frac12 \varphi u(t,x)^2 \diff t - \varphi(t,x) u(t,x) \diff x \right), \nonumber
\end{align} where
\begin{align*}
D:=\{(\tau_{n-1},x) : x \in (s(\tau_{n-1}),1)\},\, E:=\{(\tau_{n},x) : x \in (s(\tau_{n}),1)\}, \, F:=\{(t,1) : t \in (\tau_{n-1},\tau_n)\}.
\end{align*}

One can check by direct calculation that
\begin{align*}
&\sum_{n=1}^N (I_n + J_n) = - \int_{\mathbb T} \varphi(\tau_N,x) u(\tau_N,x) \diff x \\
&+ \int_0^{\tau_N} \varphi(t,s(t)) \left[u_+(t,s(t)) - u_-(t,s(t)) \right] \left( \diff s_t - \frac12 \left[u_-(t,s(t)) + u_+(t,s(t)) \right] \diff t - \xi(s(t)) \circ \diff W_t \right),
\end{align*}
where we used the assumption that $\varphi(0,\cdot) \equiv 0$.

Now, from \eqref{stoch-burger-weak}, we have $\lim_{N\rightarrow \infty}\sum_{n=1}^N (I_n + J_n) = 0$ and since $\varphi$ has compact support, there exists $N > 0$ such that $\varphi(\tau_{N'},\cdot) \equiv 0$ for all $N' \geq N$. Hence,
\begin{align}
0 &= \lim_{N \rightarrow \infty}\sum_{n=1}^N (I_n + J_n) \nonumber \\
&= \int_0^{\infty} \varphi(t,s(t)) \left[u_+(t,s(t)) - u_-(t,s(t)) \right] \left( \diff s_t - \frac12 \left[u_-(t,s(t)) + u_+(t,s(t)) \right] \diff t - \xi(s(t)) \circ \diff W_t \right), \nonumber
\end{align}
and since $\varphi$ is arbitrary, we have
\begin{align}
\diff s_t = \frac12 \left[u_-(t,s(t)) + u_+(t,s(t)) \right] \diff t + \xi(s(t)) \circ \diff W_t, \nonumber
\end{align}
for all $t > 0$.

\begin{figure}[h]
    \centering
    \includegraphics[width=7cm]{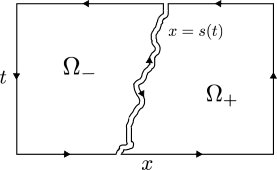}
    \caption{In the proof of the stochastic Rankine-Hugoniot condition \eqref{stoch-RH}, the domain $\Omega_n := [\tau_{n-1},\tau_n) \times (0,1) \subset [0,\infty) \times \mathbb T$ is split up into two parts: $\Omega_-^n$, which is on the left of the shock curve $(t,s(t))$ and $\Omega_+^n$, which is on the right.}
    \label{RH-figure}
\end{figure}

\end{proof}

\section{Well-posedness results}     \label{section4}
\subsection{Local well-posedness of a stochastic Burgers' equation} \label{localsection}

Now, we prove local well-posedness of the stochastic Burgers' equation \eqref{eq0} with one noise term and $\nu=0$. In fact, since the techniques used in the proof are essentially the same, we prove local well-posedness of a {\em more general equation}, which includes \eqref{eq0} as a special case. The stochastic Burgers' equation we treat is given by
\begin{equation}\label{main:eq:noise}
  \diff u + u\partial_{x}u  \diff t+ \mathcal{Q} u \diff W_{t}= \frac{1}{2} \mathcal{Q}^{2}u \diff  t.
\end{equation}
Here $\mathcal{Q}$ represents a first order differential operator \[ \mathcal{Q}u=a(x)\partial_{x}u+b(x)u, \]
where the coefficients $a(x),b(x)$ are smooth and bounded.
We state the main result of this section:
\begin{theorem}\label{mainth}
Let $u_{0}\in H^{s} (\mathbb{T})$ for some $s > 2$ fixed. Then there exists a pathwise unique $H^s$-maximal solution $(\tau_{max},u)$ of the 1D Burgers' equation \eqref{main:eq:noise} in the sense of Definition \ref{def:maximal} with initial datum $u_0$. 
Moreover, either $\tau_{max}=\infty$ or $\lim\sup_{t \rightarrow \tau_{max}} \norm{u(t)}_{H^{s}} = \infty,$ a.s.
\end{theorem}
We will provide a sketch of the proof, which follows closely the approach developed in \cite{stochBouss2018,stochEuler2017}. For clarity of exposition, let us divide the argument into several steps. \\ 
\begin{itemize}
\item {\textsl{Step 1: Uniqueness of local solutions.}  To show pathwise uniqueness of local solutions, one argues by contradiction. More concretely, one can prove that any two different solutions to (\ref{main:eq:noise}) defined up to a stopping time must coincide, as explained in the following proposition.

\begin{proposition}\label{prop:unique:local}
Let $s>2.$ Let $\tau:\Xi \rightarrow [0,\infty)$ be a stopping time, and $u^1,u^2: \Xi \times [0,\tau] \times \mathbb{T} \to \mathbb{R}$ be two $H^s$-local solutions of \eqref{main:eq:noise} with the same initial datum $u_{0} \in H^s(\mathbb{T})$. Then a.s. $u^1=u^2$ on $[0,\tau]$.
\end{proposition}

\begin{proof}
For this, we refer the reader to \cite{stochBouss2018,stochEuler2017}. It suffices to define $\bar{u} = u^1 - u^2,$ and perform standard estimates for the evolution of the $L^2$ norm of $\bar{u}$. \\
\end{proof}} 

\item{\textsl{Step 2: Existence and uniqueness of truncated maximal solutions.} Following the techniques in \cite{stochEuler2017} for the Euler equation with Lie transport noise, for $r>0$ we consider the truncated stochastic Burgers' equation
\begin{align}\label{truncated:burgers}
  \diff u_{r} + \theta_{r}(\norm{\partial_{x}u_r}_{L^\infty}) u_{r}\partial_{x}u_{r}  \diff t+ \mathcal{Q}u_{r} \diff W_{t}= \frac{1}{2} \mathcal{Q}^{2}u_{r} \diff  t,
\end{align}
where  $\theta_{r}:[0,\infty)\to [0,1]$ is a smooth function such that 
\begin{align*}
\theta_{r}(x)=
\begin{cases}
1, \quad |x| \leq r, \\
0, \quad |x| \geq 2r.
\end{cases}
\end{align*}
As we explain in the following lemma, local solutions of \eqref{main:eq:noise} can be constructed by restricting global solutions of \eqref{truncated:burgers} to a certain stopping time.

\begin{lemma}\label{globalimplieslocal}
Fix $r>0$ and $u_{0} \in H^{s}(\mathbb{T}),$ $s > 2$. Let $u_{r}:\Xi \times [0,\infty) \times \mathbb{T} \to \mathbb{R}$ be an $H^s$-global solution of \eqref{truncated:burgers} with initial datum $u_0$. Consider the stopping time
\begin{align*}
\tau_{r} = \inf \left\{t\geq 0 : \norm{u_r(t,\cdot)}_{H^{s}} \geq \dfrac{r}{C} \right\}, 
\end{align*}
where $C$ is chosen in such a way that the following inequality holds:
\begin{align*}
   \norm{\partial_x u_r}_{L^\infty} \leq C \norm{u_r}_{H^{s}}, 
\end{align*}  
which can be guaranteed thanks to the Sobolev embedding \eqref{Sob:ine2}. Then $(\tau_r,u_r)$ is a local solution to the stochastic Burgers' equation \eqref{main:eq:noise}.
\end{lemma}
\begin{proof}[Proof of Lemma \ref{globalimplieslocal}]
The proof is straightforward by construction.  For any $t\in[0,\tau_{r}],$ we have that
\begin{align*}
    & \norm{\partial_x u_r}_{L^\infty} \leq C \norm{u_r}_{H^{s}} \leq r,
\end{align*}  
and therefore $\theta_{r}(\norm{\partial_x u_r}_{L^\infty}) =1$.
\end{proof}  Let us state the result which is the {\em cornerstone} for proving existence and uniqueness of maximal local solutions of the stochastic Burgers' equation \eqref{main:eq:noise}.

\begin{proposition}\label{main:propo}
Given $r>0$ and $u_{0}\in H^{s}(\mathbb{T},\mathbb{R})$ for $s>2$, there exists a pathwise unique global $H^s$-solution $u$ of the truncated stochastic Burgers' equation \eqref{truncated:burgers}.
\end{proposition}

It is very easy to check that once Proposition \ref{main:propo} is proven, Theorem \ref{mainth} follows immediately (cf. \cite{stochEuler2017}). Therefore, we focus our efforts on showing Proposition \ref{main:propo}.} \\

\item{\textsl{Step 3: Global existence of solutions of the hyper-regularised truncated stochastic Burgers' equation.} In order to show the global well-posedness of the truncated equation, we consider the following hyper-regularisation of \eqref{truncated:burgers}
\begin{align} 
\diff  u^{\nu}_{r} + \theta_{r} (\norm{\partial_{x} u^{\nu}_{r}}_{L^\infty})u^{\nu}_{r}\partial_{x}u^{\nu}_{r} \diff t+\mathcal{Q}u^{\nu}_{r} \diff W_{t} = \nu \partial^{s'}_{xx} u^{\nu}_{r} \diff t +\frac{1}{2} \mathcal{Q}^{2}u^{\nu}_{r} \diff t,  \label{hyper-regularised-eq}
\end{align}
where $\nu>0$ is a positive parameter and $s'= 2[s]+1$. Notice that we have added dissipation in order to be able to perform our computations rigorously. Equation \eqref{hyper-regularised-eq} is understood in the {\em mild sense}, i.e., as a solution to an integro-differential equation, which we discuss in the next step (see \eqref{mild-eq}).

\begin{proposition}\label{prop:global:hyper}
For every $\nu,r >0$ and initial datum $u_{0}\in H^{s}(\mathbb{T})$ for $s > 2$, there exists a pathwise unique global strong solution $u^{\nu}_{r}$ of (\ref{hyper-regularised-eq}) in the class $L^{2}(\Xi ; C([0,T]; H^{s}(\mathbb{T})))$, for all $T>0$. Moreover, its paths gain instant extra regularity $C([\delta,T]; H^{s+2}(\mathbb{T}) )$, for every $T>\delta>0$. 
\end{proposition}
\begin{proof}[Proof of Proposition \ref{prop:global:hyper}]
The proof is based on a classical fixed point iteration argument which employs Duhamel's principle (see \cite{stochEuler2017}).  We omit the subscripts $\nu$ and $r$ throughout the proof for simplicity, but it should be kept in mind that our functions depend on those parameters. Given $u_{0}\in H^{s}(\mathbb{T})$, consider the mild formulation of the hyper-regularised truncated equation \eqref{hyper-regularised-eq}
\begin{align}
    u(t)= (\Upsilon u)(t), \label{mild-eq}
\end{align}
where
\begin{align*} 
& (\Upsilon u)(t)= e^{tA}u_{0} -\int
_{0}^{t}e^{(  t-s)  A}W_{\theta} u (s)
\diff  s \\
& \hspace{100pt} +\int_{0}^{t}e^{(  t-s)  A} Lu(s) \diff  s - \int_{0}^{t}e^{(  t-s)  A}Ru(s) \diff W_{s},
\end{align*}
$t>0$ and we have employed the notation $A=\nu \partial^{s'}_{xx},$ $W_{\theta}u=\theta_{r} (\norm{\partial_{x} u}_{L^\infty}) u\partial_{x}u,$
$Lu=\dfrac{1}{2}\mathcal{Q}^{2}u,$ and $Ru= \mathcal{Q}u$. Define the space $\mathcal{W}_{T}=L^{2}(\Xi; C([0,T];H^{s}(\mathbb{T})))$.  One can show that $\Upsilon$ is a contraction on $\mathcal{W}_{T}$ by following classical arguments as in \cite{stochEuler2017}. Therefore, by applying Picard's iteration, a local solution can be constructed. To extend it to a global one, it is sufficient to show that for any given $T>0$ and initial datum $u_{0}\in H^{s}(\mathbb{T})$, $s > 2,$ the following bound is available 
\begin{align}\label{eq:assertion}
\mathbb{E}\left[ \sup_{t\in[0,T]} \norm{u(t)}^{2}_{H^{s}} \right] \leq C(T),
\end{align}
for a finite constant $C(T) < \infty,$ so that one can patch together each local solution to cover any finite time interval $[0,T]$. We will prove estimate \eqref{eq:assertion} further below. Furthermore, by standard properties of the semigroup $e^{tA}$ (cf. \cite{Goldstein85}), one can prove that for positive times $T>\delta>0$, each term in the mild equation \eqref{mild-eq} enjoys higher regularity, namely, $u\in L^{2}(\Xi; C([\delta,T];H^{s+2}(\mathbb{T})))$. All the computations are omitted and can be carried out easily by mimicking the same ideas as in \cite{stochEuler2017,stochBouss2018}.  
\end{proof}}

\item{\textsl{Step 4: Limiting and compactness argument.} The main objective of this step is to show that the family of solutions $\{u^{\nu}_{r} \}_{\nu>0}$ of the hyper-regularised stochastic Burgers' equation \eqref{hyper-regularised-eq} is {\em compact} in a particular sense and therefore we are able to extract a subsequence converging strongly to a solution of the truncated stochastic Burgers' equation \eqref{truncated:burgers} in a convenient space. The central idea for proving this is to show compactness
of the probability laws of this family. Consequently, we demonstrate that these laws are {\em tight} in a suitable metric space.  Let $T>0$ and define the Polish space $E$ by
\begin{equation}\label{polish}
E = C([0,T];H^{\beta}(\mathbb{T})), \quad \beta \geq 2.
\end{equation}
Assume that the laws of $\{u^\nu_r\}_{\nu>0}$ are tight in $E.$ Once this is proven, one only needs to invoke standard stochastic partial differential equations arguments based on the Skorokhod's representation and Prokhorov's theorem to conclude that there exists a subsequence of $\{u^\nu_r\}_{\nu>0}$ such that solutions of equation \eqref{hyper-regularised-eq} converge to solutions of \eqref{truncated:burgers} in the weak limit in the Polish space $E$ \eqref{polish}. A more thorough approach can be found in \cite{stochEuler2017,GlattVicol2014}. In the next proposition, we present the main argument to show that the sequence of laws are indeed tight.

\begin{proposition} \label{tightness}
Assume that for some $\alpha >0,$ $M \in \mathbb{N},$ there exist  constants $C_1(T)$ and $C_2(T)$ such that
\begin{align}
\mathbb{E} \left [ \sup_{t \in [0,T]} \norm{u_r^\nu (t)}^4_{H^{s}} \right ] \leq C_1(T), \label{tight1}
\end{align}
\begin{align}
\mathbb{E} \left [ \int_0^T \int_0^T \frac{\norm{u_r^\nu (t) - u_r^\nu(s)}^4_{H^{-M}}}{|t-s|^{1+4\alpha}} \diff t \diff s \right ]  \leq C_2(T), \label{tight2}
\end{align}
uniformly in $\nu$. Then the sequence $\{u_{r}^{\nu}\}_{\nu>0}$ is tight in $E$.
\end{proposition}

\begin{proof}[Proof of Proposition \ref{tightness}]
We employ the following lemma, which can be found in \cite{stochEuler2017}, was originally proved in \cite{simon}, and constitutes a variation of the classical Aubin-Lions Lemma.
\begin{lemma}\label{compactlemma}
Suppose that $X, Y, Z$  are separable Hilbert spaces with continuous dense embedding $X\hookrightarrow Y\hookrightarrow Z$ such that there exists $\theta \in (0,1)$ and $M >0$ verifying 
\begin{align*}
    \norm{v}_Y \leq M \norm{v}_X^{1-\theta} \norm{v}_Z^{\theta},
\end{align*}
for all $v \in X.$ Assume that $X \hookrightarrow Y$ is a compact embedding. Let $\alpha>0.$ Then 
\begin{align*}
L^{\infty}([0,T]; X) \cap W^{\alpha,4}([0,T]; Z)\hookrightarrow C([0,T]; Y) 
\end{align*}
is a compact embedding.
\end{lemma}
In Lemma \ref{compactlemma} we select
\begin{align*} 
X=H^{s}(\mathbb{T}), \quad Y=H^{\beta}(\mathbb{T}), \quad Z= H^{-M}(\mathbb{T}),
\end{align*}
where $\beta \geq 2$ as specified before and we impose the extra condition $\beta < s$ so that the embedding of $X$ into $Y$ is compact. We also choose $\alpha \in (0,1)$. Therefore we obtain that $L^\infty([0,T]; H^{s}(\mathbb{T})) \cap W^{\alpha,4}([0,T]; H^{-M}(\mathbb{T}))$ is compactly embedded in $E$. By hypotheses \eqref{tight1}-\eqref{tight2} \cite{stochEuler2017}, the family of laws of $\{u^{\nu}_{r}\}_{\nu>0}$ is $\mathbb{P}$-a.s. supported on the space 
\begin{align*}
E_0 = L^\infty([0,T]; H^{s}(\mathbb{T})) \cap W^{\alpha,4}([0,T]; H^{-M}(\mathbb{T})),
\end{align*}
and it suffices to prove that this family is tight in $E_0$. For $A_1,A_2,A_3>0,$ define the set
\begin{align*}
& B = \bigg\{ f:[0,T] \times \mathbb{T} \rightarrow \mathbb R \hspace{0.2cm}: \sup_{t \in [0,T]} \norm{f(t)}_{H^{s}}^4  \leq A_1, \int_0^T \norm{f(t)}_{H^{-M}}^4 \diff t \leq A_2, \\ 
& \hspace{170pt} \int_0^T \int_0^T \frac{\norm{f(t)-f(s)}^4_{H^{-M}}}{|t-s|^{1+4\alpha}} \diff t \diff s \leq A_3  \bigg\},
\end{align*}
which is relatively compact in $E_0$. It is enough to show that for every $\epsilon,$ there exist $A_{1},A_2, A_{3}>0$ such that $ \mathbb{P} (u_r^\nu \in B^{c}) \leq \epsilon.$ Fix $\epsilon>0.$ Invoking Markov's inequality and taking into account hypothesis \eqref{tight1}, we have that
\begin{align*}
\mathbb{P} \left( \sup_{t \in [0,T]} \norm{u_r^\nu(t)}^4_{H^{s}} > A_1 \right) \leq  \frac{\mathbb{E} \left [ \sup_{t \in [0,T]} \norm{u_{r}^\nu (t)}^4_{H^{s}} \right] }{A_1} \leq \frac{C_1(T)}{A_1},
\end{align*}
and this is smaller than $\epsilon/3$ if we choose $A_{1}$ sufficiently large. 
Moreover, since $\norm{f}_{H^{-M}} \lesssim \norm{f}_{H^{s}},$ 
\begin{align*}
\mathbb{P} \left( \int_0^T \norm{u_r^\nu(t)}^4_{H^{-M}} \diff t > A_2  \right) \leq  \mathbb{P} \left( T \sup_{t\in[0,T]} \norm{u_r^\nu(t)}^4_{H^{-M}} > A_2  \right) \\ 
\lesssim   \mathbb{P} \left( T \sup_{t\in[0,T]} \norm{u_r^\nu(t)}^4_{H^{s}} > A_2  \right)\leq \frac{C_1(T) T}{A_{2}},
\end{align*}
which can also be made arbitrarily small. A similar argument applies to the set
\begin{align*}
    \mathbb{P} \left( \int_0^T \int_0^T \frac{\norm{u_r^\nu(t)-u_r^\nu(s)}_{H^{-M}}^4}{|t-s|^{1+4\alpha}} \diff t \diff s > A_3 \right),
\end{align*} 
by using hypothesis \eqref{tight2}. Hence we conclude that $\mathbb{P} (u_r^\nu \in B^{c}) \leq \epsilon$ if we choose $A_1,A_2,A_3$ large enough, as desired.
\end{proof}}

\item{\textsl{Step 5: Hypothesis estimates.}}
We are left to show that hypothesis \eqref{tight1}-\eqref{tight2} hold.  First, we prove that condition \eqref{tight1} implies condition \eqref{tight2}. By following the techniques in \cite{stochEuler2017} and using equation \eqref{hyper-regularised-eq}, one can obtain
\begin{align*}
& \mathbb{E} \left[ \norm{u^{\nu}_{r}(t)-u^{\nu}_{r}(s)}^4_{H^{-M}} \right]
\lesssim  (t-s)^3 \int_s^t \mathbb{E} \left[\theta_{r} (\norm{\partial_{x} u^{\nu}_r}_{L^\infty})^4\norm{u^{\nu}_{r}\partial_{x}u^{\nu}_{r}}^4_{H^{-M}}\right] \diff \gamma \\ 
& + (t-s)^3 \int_s^t \mathbb{E} \left[ \nu \norm{ \partial_{xx}^{s'} u^{\nu}_{r} }^4_{H^{-M}} \right] \diff \gamma + (t-s)^3 \int_s^t \mathbb{E} \left[ \norm{ \mathcal{Q}^{2}u^{\nu}_{r} }^4_{H^{-1}} \right] \diff \gamma \\
& \hspace{150pt} + \mathbb{E} \left[ \left| \left| \int_s^t \mathcal{Q} u^{\nu}_{r} \diff W_{\gamma} \right| \right|^4_{L^2} \right],
\end{align*}
for $0\leq s < t \leq T.$ It is easy to infer that
\begin{align*}
 \int_s^t \mathbb{E} \left[\theta_{r}(\norm{\partial_{x}u_r^\nu}_{L^\infty})^4 \norm{u^{\nu}_{r}\partial_{x}u^{\nu}_{r}}^{4}_{H^{-M}}\right] \diff \gamma \lesssim  \int_{s}^{t} \mathbb{E}\left[\norm{u^{\nu}_{r}}^{4}_{H^{s}}\right] \diff \gamma \leq A_1(t-s),
\end{align*}
since
\begin{align*}
\norm{u^{\nu}_{r}\partial_{x}u^{\nu}_{r}}_{H^{-M}}  \lesssim \norm{\partial_{x}u_r^\nu}_{L^\infty} \norm{u_r^\nu}_{H^{s}},
\end{align*}
where we have taken into account hypothesis \eqref{tight1}. In a similar way, one can check that for $M=3[s']+2$, 
\begin{align*}
\int_s^t \mathbb{E} \left[\norm{\nu \partial^{s'}_{xx} u^{\nu}_{r} }^4_{H^{-M}} \right] \diff \gamma \lesssim  \int_s^t \mathbb{E}\left[\norm{u_r^\nu }^{4}_{H^s}\right] \diff \gamma \leq A_2(t-s),
\end{align*}
since 
\begin{align*} \norm{\partial^{s'}_{xx} u_r^\nu}_{H^{-M}}  \lesssim  \norm{u_r^\nu}_{H^s}. 
\end{align*}
By similar techniques, we have 
\begin{align*}
\int_{s}^{t}\mathbb{E}\left[ \norm{\mathcal{Q}^{2}u^{\nu}_{r} }^4_{H^{-1}}\right] \diff \gamma \leq A_3(t-s),
\end{align*}
and (take into account Burkholder-Davis-Gundy as in \cite{stochEuler2017})
\begin{align*}
\mathbb{E} \left[ \left| \left| \int_s^t \mathcal{Q}u^{\nu}_{r} \diff W_{\gamma} \right| \right|^4_{L^2} \right] \leq A_4(t-s)^2.
\end{align*}
Combining the above estimates, we deduce that
\begin{align*} \mathbb{E} \left[ \norm{u^{\nu}_{r}(t)-u^{\nu}_{r}(s)}^4_{H^{-M}} \right] \leq A_5(T) (t-s)^2.\end{align*}
Hence for $0<\alpha < 1/2,$ 
\begin{align*}
\mathbb{E} \left[ \int_0^T \int_0^T \frac{\norm{u_r^\nu (t) - u_r^\nu(s)}^4_{H^{-M}}}{|t-s|^{1+4\alpha}} \diff t \diff s \right] 
&\leq  \mathbb{E} \left[ \int_0^T \int_0^T \frac{A_5(T)}{|t-s|^{4 \alpha-1}} \diff t \diff s \right] \\ &\leq C_1(T).
\end{align*}
We are left to prove that  hypothesis \eqref{tight1} holds, i.e.
\begin{align}\label{est:apriori:full}
\mathbb{E} \left [ \sup_{t \in [0,T]} \norm{u_r^\nu (t)}^4_{H^{s}} \right ] \leq C_2(T).
\end{align}
This is the most difficult part of the hypotheses estimates. We compute the evolution of the $\dot{H}^s$-norm of $u_r^\nu$. Indeed, the equation for the $L^{2}$-norm of $\Lambda^{s}u_r^\nu$ in \eqref{hyper-regularised-eq} is
\begin{align} \label{hacheese}
& \frac{1}{2} |\Lambda^{s} u^{\nu}_{r} (t)|^2_{L^2} = \frac{1}{2} |\Lambda^{s} u^{\nu}_{r} (0)|^2_{L^2} \nonumber \\ 
&\hspace{50pt} - \int_0^t \langle \theta_{r}(\norm{\partial_{x} u_r^\nu }_{L^\infty}) \Lambda^{s} (u^{\nu}_{r}\partial_{x}u^{\nu}_{r} ), \Lambda^{s} u^\nu_r \rangle_{L^2} \diff \gamma \nonumber \\ 
&\hspace{50pt} - \int_0^t \langle \Lambda^{s} \mathcal{Q} u^{\nu}_{r}, \Lambda^{s} u_r^\nu \rangle_{L^2} \diff W_{\gamma} 
 +  \int_0^t \langle \nu \Lambda^{s} \partial^{s'}_{xx} u^{\nu}_{r}, \Lambda^{s} u^\nu_r \rangle_{L^2} \diff \gamma \nonumber  \\ 
& \hspace{50pt} + \frac{1}{2} \int_0^t  \langle \Lambda^{s} \mathcal{Q}^{2} u^{\nu}_{r}, \Lambda^{s} u^\nu_r \rangle_{L^2} \diff \gamma  +\frac{1}{2}  \int_0^t \langle \Lambda^{s} \mathcal{Q} u^\nu_r, \Lambda^{s} \mathcal{Q}u^\nu_r \rangle_{L^2} \diff \gamma,
\end{align}
for $t \in [0,T].$ The nonlinear term can be estimated as
\begin{align*}
&  \int_{\mathbb{T}}  \Lambda^{s} (u^{\nu}_{r}\partial_{x}u^{\nu}_{r})\Lambda^{s} u^\nu_r \diff V  =  \int_{\mathbb{T}}  (u^{\nu}_{r} \Lambda^{s}  \partial_{x} u^{\nu}_{r})\Lambda^{s} u^\nu_r \diff V  + \int_{\mathbb{T}}  [\Lambda^{s},  u^{\nu}_{r} ] \partial_x u^\nu_r \Lambda^{s} u^\nu_r \diff V   \nonumber \\
& \hspace{250pt} = I_1 + I_2.  
\end{align*}
Integrating by parts, we note that $I_1$ can be rewritten and bounded as
\begin{align*}
    |I_1| = \frac{1}{2}\left| \int_{\mathbb{T}}  \partial_{x}  u^{\nu}_{r}  (\Lambda^{s} u^{\nu}_{r})^2 \diff V \right| \leq \norm{\partial_x u^{\nu}_{r}}_{L^\infty} \norm{\Lambda^s u^{\nu}_{r}}_{L^2}^2.
\end{align*}
$I_2$ can be estimated via the Kato-Ponce commutator estimate \eqref{katoponce}
\begin{align*}
    &|I_2| \leq  \norm{[\Lambda^{s},  u^{\nu}_{r} ] \partial_x u^{\nu}_{r}}_{L^2} \norm{\Lambda^{s} u^\nu_r}_{L^2} \\ 
    & \lesssim ( \norm{\partial_x u^{\nu}_{r}}_{L^\infty} \norm{\Lambda^{s-1} \partial_x u^{\nu}_{r}}_{L^2} + \norm{\Lambda^s u^{\nu}_{r}}_{L^2}   \norm{\partial_x u^{\nu}_{r}}_{L^\infty})  \norm{\Lambda^{s} u^\nu_r}_{L^2} \\
    & \hspace{150pt} \lesssim \norm{\partial_x u^{\nu}_{r}}_{L^\infty} \norm{\Lambda^s u^{\nu}_{r}}_{L^2}^2.
\end{align*}
Putting together the estimates for $I_1$ and $I_2,$ we get
\begin{align*} 
    \left| \int_{\mathbb{T}}  \Lambda^{s} (u^{\nu}_{r}\partial_{x}u^{\nu}_{r})\Lambda^{s} u^\nu_r \diff V \right| \lesssim \norm{\partial_x u^{\nu}_{r}}_{L^\infty} \norm{\Lambda^s u^{\nu}_{r}}_{L^2}^2.
\end{align*}
We note that the dissipative term
\begin{align*} \langle \nu \Lambda^{s} \partial^{s'}_{xx} u^{\nu}_{r}, \Lambda^{s} u^\nu_r \rangle_{L^2}= -\nu \norm{\Lambda^{s+2[s]+1}u^{\nu}_{r}}^{2}_{L^{2}} \leq 0 
\end{align*}
has the ``good" sign so it can be dropped. The last two terms in \eqref{hacheese} can be treated in a similar fashion by taking into account estimate \eqref{eq:cancellation2:thm}. The previous techniques also serve to treat the evolution of $\norm{u^\nu_r }_{L^2}$ (this is a much simpler case).\\ \\
Denote
\begin{align*}
M_t = \int_0^t ( \langle \mathcal{Q} u^{\nu}_{r},  u_r^\nu \rangle_{L^2} + \langle \Lambda^{s} \mathcal{Q} u^{\nu}_{r}, \Lambda^{s} u_r^\nu \rangle_{L^2} ) \diff W_\gamma, \quad t \in [0,T].
\end{align*}
We drop the parameter dependence on $u_{r}^{\nu}$ to make the notation simpler. By putting together the previous estimates, we have
\begin{align*} \norm{u(t)}^{2}_{H^{s}} \lesssim \norm{u(0)}^{2}_{H^{s}} +  |M_t| + \int_0^t \norm{u}^{2}_{H^{s}} \diff \gamma,
\end{align*}
where the constant in the last inequality depends on $r$. Integrating the above expression against $\exp(t)$ and squaring as in \cite{stochEuler2017}, we obtain
\begin{align*}
\norm{u(t)}_{H^{s}}^4  \lesssim \text{exp} (t) ( \norm{u(0)}^4_{H^{s}}+ |M_t|^2 ). 
\end{align*}
By applying supremum and expectation on both sides of the equation above, we have
\begin{align}\label{est:local:final}
\mathbb{E}  \left[ \sup_{\gamma \in [0,t]}\norm{u(\gamma)}_{H^{s}}^4 \right] \lesssim \text{exp} (t) \left( \norm{u(0)}^4_{H^{s}}+ \mathbb{E} \left[ \sup_{\gamma \in [0,t]} |M_\gamma|^2 \right]\right),
\end{align}
where we remind that $t \in [0,T].$ We apply Burkholder-Davis-Gundy inequality \eqref{BGD:ineq} in order to obtain the bound
\begin{align}\label{est:burk-dav-gun}
\mathbb{E} \left[ \sup_{\gamma \in [0,t]} |M_{\gamma}|^{2} \right] \lesssim  \mathbb{E} [[M]_{t}], 
\end{align}
where
\begin{align*}
    [M]_t = \int_0^t ( \langle \mathcal{Q} u,  u \rangle_{L^2} + \langle \Lambda^{s} \mathcal{Q} u, \Lambda^{s} u \rangle_{L^2} )^2 \diff \gamma .
\end{align*}
Integrating by parts, we can derive
\begin{align}\label{est:easypart:cancellations}
|\langle \Lambda^{s}  \mathcal{Q}u, \Lambda^{s} u \rangle_{L^2} |
\lesssim  \norm{u}^{2}_{H^{s}},
\end{align}
as in the proof of \eqref{eq:cancellation2:thm} in \cite{stochBouss2018} or in the appendix of \cite{stochBouss2018}. The constant in the last inequality depends on the $W^{s+1,\infty}$-norm of the coefficients. Therefore,
\begin{align}\label{est:martingale}
\mathbb{E} [[M]_t] \lesssim \int_0^t \mathbb{E} \left[ \sup_{r \in [0,t] } \norm{ u (r) }^{4}_{H^{s}}\right] \diff \gamma.
\end{align}
Hence, combining \eqref{est:local:final}-\eqref{est:martingale}, together with application of Gr\"{o}nwall's inequality yields 
\begin{align*}
\mathbb{E} \left [ \sup_{t \in [0,T]} \norm{u (t)}^4_{H^{s}} \right ] \leq C_2(T).
\end{align*}
\end{itemize}

\subsection{Blow-up criterion}
We are now interested in deriving a blow-up criterion for the stochastic Burgers' equation (\ref{eq0}) with $\nu=0$. First of all, we note that for the deterministic Burgers' equation
\begin{align} \label{burgerdet}
u_t + u \partial_x u = 0,
\end{align}
there is a well-known blow-up criterion available. In the deterministic case, the following theorem of local existence and uniqueness of strong solutions holds.
\begin{theorem}
Let $u_0 \in H^s(\mathbb{T})$ with $s > 3/2.$ Then there exists $T>0$ and $u \in C([0,T];H^s(\mathbb{T})) \cap C^1([0,T];H^{s-1}(\mathbb{T}))$ solving equation \eqref{burgerdet}. Moreover, if $v$ is another function with the same properties, necessarily $u=v.$ 
\end{theorem}
The above result is classic and numerous proofs are available in the literature (see \cite{KisNazSht2008} for a related cutting edge result and the references therein). The blow-up criterion for equation \eqref{burgerdet} is the following.
\begin{theorem}[Blow-up criterion for deterministic Burgers'] \label{blowthe}
Assume that $u_0 \in H^s(\mathbb{T}),$ $s >3/2,$ $T^*>0,$ and $u:[0,T^*) \times \mathbb{T} \rightarrow \mathbb R$ is a local solution of \eqref{burgerdet}. Then the following statements are equivalent
\begin{itemize}
    \item $\lim_{t \rightarrow T^*} \norm{u(t,\cdot)}_{H^s} = \infty.$
    \item $\int_{0}^{T^{\star}} \norm{\partial_x u(t,\cdot)}_{L^\infty} \diff t = \infty.$
\end{itemize}

\end{theorem}
In the rest of this subsection, we focus on proving the following stochastic version of Theorem \ref{blowthe}.
\begin{theorem}[Blow-up criterion for stochastic Burgers'] \label{blowup}
Assume that $u_0 \in H^s(\mathbb{T}),$ $s > 2,$ and $u:\Xi \times [0,\tau_{max}) \times \mathbb{T} \rightarrow \mathbb R$ is a maximal solution of \eqref{main:eq:noise}. If $\tau_{max}<\infty,$ then
$\int_0^{\tau_{max}}  \norm{\partial_x u(t, \cdot)}_{L^\infty }\diff t=\infty$ a.s. Moreover, if $0\leq\tau<\tau_{max}$ is a smaller stopping time, necessarily $\int_0^{\tau}  \norm{\partial_x u(t, \cdot)}_{L^\infty} \diff t<\infty$ a.s.
\end{theorem}

\begin{proof}[Proof of Theorem \ref{blowup}] 
By following the argument in \cite{stochEuler2017}, we start by noting that it is clear that if $0\leq\tau<\tau_{max}$ is a smaller stopping time, necessarily $\int_0^{\tau}  \norm{\partial_x u(t, \cdot)}_{L^\infty} \diff t <\infty$ a.s. This is guaranteed by the embedding
\begin{align} \label{tipico}
\norm{\partial_x u}_{L^\infty}  \lesssim \norm{u}_{H^s}.
\end{align}
Proving $\int_0^{\tau_{max}}  \norm{\partial_x u(t, \cdot)}_{L^\infty} \diff t=\infty$ is more involved. For this, we consider the hyper-regularised truncated Burgers' equation (\ref{hyper-regularised-eq}) and define the following stopping times:
\begin{align*}
\tau^s = \lim_{n \rightarrow \infty} \tau_n^s, \qquad \tau_n^s = \inf \left\{t \geq 0: \norm{u_r^\nu(t, \cdot)}_{H^s} \geq n  \right\},
\end{align*}
\begin{align*}
\tau^{\infty} = \lim_{n \rightarrow \infty} \tau_n^{\infty}, \qquad \tau_n^{\infty} = \inf \left\{t \geq 0: \int_0^t  \norm{\partial_x u_r^\nu(s, \cdot)}_{L^\infty}   \diff s\geq n  \right\}.
\end{align*}
We claim that $\tau^\infty = \tau^s.$ By again making use of \eqref{tipico}, it is easy to check $\tau^s\leq \tau^\infty.$ For simplicity, we omit the subscripts $\nu$ and $r$ throughout the proof. 
Imitating the techniques from the previous subsection, we arrive at
\begin{align*}
& \frac{1}{2} \diff  \norm{\Lambda^s u}^{2}_{L^{2}}  + \theta(\norm{\partial_{x}u}_{L^\infty}) \langle \Lambda^s (u \partial_{x}u) , \Lambda^s u \rangle_{L^{2}} \diff t +  \langle \Lambda^s \mathcal{Q} u, \Lambda^s u \rangle_{L^{2}} \diff W_{t} \\
& = \nu \langle \Lambda^{s+2s'} u, \Lambda^s u \rangle_{L^{2}} \diff t +\frac{1}{2}  \langle \Lambda^s \mathcal{Q}^2  u, \Lambda^s u \rangle_{L^2} \diff t + \frac{1}{2}  \langle \Lambda^s \mathcal{Q} u, \Lambda^s \mathcal{Q} u \rangle_{L^2} \diff t.
\end{align*}
By integration by parts and standard estimates (use \eqref{eq:cancellation2:thm} for estimating the last two terms), one gets
\begin{align*} 
\diff \norm{\Lambda^s u}^{2}_{L^{2}} + 2  \langle \Lambda^s \mathcal{Q} u, \Lambda^s u \rangle_{L^{2}} \diff W_{t} \lesssim ( 1 + \norm{\partial_{x}u}_{L^\infty} ) \norm{\Lambda^s u}^{2}_{L^{2}}  \diff t.
\end{align*}
By also deriving a similar estimate for $s=0$ (which follows in a simpler manner by application of the techniques in the previous subsection), we obtain
\begin{align} \label{blowestimada}
\diff  \norm{u}^{2}_{H^{s}}  + 2 (\langle \mathcal{Q} u,u \rangle_{L^2} + \langle \Lambda^s \mathcal{Q} u,\Lambda^s u \rangle_{L^2}  )  \diff W_{t}  \lesssim ( 1 + \norm{\partial_{x}u}_{L^\infty} )\norm{u}^{2}_{H^{s}} \diff t.
\end{align}
Finally, we treat the stochastic term. Without loss of generality, we assume $\norm{u}_{H^s} \geq \epsilon> 0$ (otherwise add a positive constant to this function) and by applying It\^{o}'s formula in $L^2$ \cite{rozovskii} to the logarithm, we obtain
\begin{align}\label{stochasticitoformula}
\diff  \text{log} ( \norm{u}^{2}_{H^s} )= \frac{\diff  \norm{u}^{2}_{H^s} }{\norm{u}^{2}_{H^s}}-\frac{\diff  N_{t}}{2 (\norm{u}^{2}_{H^s})^{2}},
\end{align}
where $N_t$ is defined by
\begin{align*}N_{t}= 4 \int_{0}^{t} ( \langle \mathcal{Q} u, u \rangle_{L^2} +  \langle \Lambda^s \mathcal{Q} u, \Lambda^s u  \rangle_{L^2}  )^2 \diff \gamma.
\end{align*}
Notice that we have the bound  
\begin{align} \label{surprising}
    |N_t| \lesssim \norm{u}_{H^s}^2,  
\end{align}
as we have already indicated before (see the proof of \eqref{eq:cancellation2:thm} in the appendix of \cite{stochBouss2018}), where the constant in the last inequality depends on the $W^{s+1,\infty}$-norm of the coefficients of $Q$. Plug \eqref{blowestimada} and \eqref{surprising} into \eqref{stochasticitoformula} to derive the estimate
\begin{align} \label{sinint} 
\diff \text{log} (\norm{u}^{2}_{H^s} ) \lesssim  \frac{(1+\norm{\partial_{x}u}_{L^\infty} ) \norm{u}^{2}_{H^s}}{\norm{u}^{2}_{H^s}} \diff t+ \diff M_{t}, \end{align}
where $M_t$ is the local martingale 
\begin{align*} M_t = \int_0^t \frac{ \langle \mathcal{Q} u, u  \rangle_{L^2} + \langle \Lambda^s \mathcal{Q} u, \Lambda^s u  \rangle_{L^2} }{\norm{u}^{2}_{H^s}} \diff W_{\gamma},
\end{align*}
for any $t>0.$ Expressing \eqref{sinint} in integral form, we have
\begin{align}    \label{eq:gronwallform}
\text{log} (\norm{u(t)}^{2}_{H^s} ) \lesssim  \text{log} (\norm{u(0)}^{2}_{H^s}) + \int_0^t (1+\norm{\partial_{x}u}_{L^\infty} ) \diff \gamma  + \int_{0}^{t} \diff M_{\gamma}.
\end{align}
By applying Burkholder-Davis-Gundy inequality (see \eqref{BGD:ineq}), we can control the local martingale term by estimating
\begin{align*}
& [ M ]_{t}  = \int_{0}^{t} \frac{(\langle \mathcal{Q} u, u \rangle_{L^2} + \langle \Lambda^s \mathcal{Q} u, \Lambda^s u \rangle_{L^2} )^{2}}{\norm{u}^{4}_{H^s} } \diff \gamma \lesssim t.
\end{align*}
Here, we have employed again the arguments in the appendix of \cite{stochBouss2018} to bound the numerator in the fraction above. Burkholder-Davis-Gundy then yields
\begin{align}\label{martingale:est:2}
\mathbb{E} \left[ \sup_{s\in[0,t]} \left|\int_{0}^{s} \diff M_{\gamma}  \right| \right] \lesssim \sqrt{t}. 
\end{align}
Taking expectation in \eqref{eq:gronwallform} and using estimate \eqref{martingale:est:2}, we establish that for any $n, m \in \mathbb{N},$
\begin{align*}
\mathbb{E}\left[\sup_{\gamma \in[0,\tau^{\infty}_{n}\wedge m ]} \text{log} (\norm{u(\gamma)}^{2}_{H^{s}} ) \right] \lesssim \text{log} (\norm{u_{0}}^{2}_{H^{s}} ) + m(1+n) + \sqrt{m} < \infty.
\end{align*}
Therefore 
\begin{align*}
\mathbb{E} \left[  \text{log} \left(\sup_{\gamma \in[0,\tau^{\infty}_{n}\wedge m ]} \norm{u(\gamma)}_{H^s}^2 \right) \right]< \infty,
\end{align*}
which in particular means that for $n,m \in \mathbb{N},$ $\sup_{\gamma \in[0,\tau^{\infty}_{n}\wedge m ]} \norm{u(\gamma)}_{H^s}$ is finite a.s. Hence
\begin{align} \label{ultima}
\mathbb{P}\left( \sup_{\gamma \in[0,\tau^{\infty}_{n}\wedge m ]} \norm{u(\gamma)}^{2}_{H^{s}} < \infty \right) =1, 
\end{align}
for every $n,m\in \mathbb{N}$, which from \cite{stochEuler2017} implies $\tau^{\infty}\leq \tau^{s}$. Therefore $\tau^\infty = \tau^s.$ Recall that we have omitted subscripts throughout this proof but $u = u^{\nu}_{r}$. By application of Fatou's Lemma we also obtain \eqref{ultima} in the limit $\nu \rightarrow 0,$ $r \rightarrow \infty,$ and we can show that $\tau^\infty=\tau^s$ also holds in the limit, concluding our argument. The few steps missing can be checked in \cite{stochEuler2017}.
\end{proof}

\subsection{Global well-posedness of a viscous stochastic Burgers' equation} 
The {\em viscous} stochastic Burgers' equation is given by
\begin{align} \label{viscosidad}
\diff u + u \partial_x u \diff t + \mathcal{Q} u \diff W_t = \frac{1}{2} Q^2 u \diff t + \nu \partial_{xx}u \diff t.
\end{align}
We assume $Q = \xi(x)\partial_x$, $s>2.$
\begin{remark}
Observe that from our assumption on $Q$, \eqref{viscosidad} is simply \eqref{eq0} with one noise term, but we wish to keep the $Q$-notation for convenience.
\end{remark}
We establish the global well-posedness of strong solutions of \eqref{viscosidad}. More concretely, we prove the following theorem.

\begin{theorem}\label{mainvisc}
Let $u_{0} \in H^{s} (\mathbb{T})$ for some $s> 2.$ Then there exists a pathwise unique maximal $H^s$-global solution of the viscous stochastic Burgers' equation \eqref{viscosidad}.
\end{theorem}
The proof follows the same strategy as the local existence proof for the Burgers' equation without viscosity \eqref{main:eq:noise} that we provided in Subsection \ref{localsection}. The most important part is proving the following a priori estimate. In this estimate, we assume that $u$ is smooth enough, but the rigorous way to do this is to regularise the equation first as we did in Subsection \ref{localsection}. 
\begin{proposition} \label{maximum1}
Let $u_{0}\in H^{s}(\mathbb{T}),$ $T>0,$ and $u: \Xi \times [0,T] \times \mathbb{T} \rightarrow \mathbb R$ be a solution to equation \eqref{viscosidad} that we assume to be smooth enough. Then there exists a constant such that \footnote{The attentive reader might ask why the $H^s$-norm of the velocity in \eqref{maximum1eq} is taken to the power of two instead of four as in \eqref{tight1}, but the arguments we carry out together with the control we provide for the $L^\infty$-norm of $u$ immediately yield the bound for the fourth power.}
\begin{align}\label{maximum1eq}
\quad \mathbb{E} \left [ \sup_{t \in [0,T]} \norm{u (t)}^2_{H^{s}} \right ] \leq C_*(T).
\end{align} 
\end{proposition} 
Once the a priori estimate \eqref{maximum1eq} is established, one can repeat the arguments in Subsection \ref{localsection} to obtain Theorem \ref{mainvisc}. However, since this is repetitive and tedious, we do not explicitly carry out these arguments here. From now on, we focus on proving Proposition \ref{maximum1}.

\begin{proof} We start by computing the evolution of the $L^{2}$-norm of the solution $u$. First note that the viscosity term has the good sign since
\begin{align*} \nu  \langle \partial_{xx} u , u  \rangle_{L^2} = - \nu \norm{\partial_{x}u}^{2}_{L^2}. \end{align*} 
By applying the same techniques as in Subsection \ref{localsection} (take into account estimate (\ref{eq:cancellation1:thm})), we obtain
\begin{align*}
\diff \norm{u}_{L^2}^2 + 2  \langle \mathcal{Q} u, u \rangle_{L^2} \diff W_t + 2 \nu \norm{\partial_{x}u}^{2}_{L^2} \diff t \lesssim \norm{u}^{2}_{L^2} \diff t ,
\end{align*}
and we have the bounds
\begin{align} \label{para1}
& \mathbb{E} \left[ \sup_{t \in [0,T]} \norm{u(t,\cdot)}^{2}_{L^{2}} \right] \leq C_1(T), \\ \label{para2}  
& \mathbb{E} \left[ \int_{0}^{T} \norm{\partial_{x}u(s, \cdot)}^{2}_{L^{2}}  \diff s \right ] \leq C_2(T). 
\end{align}
We compute the evolution of the $\dot{H}^s$-norm of $u$
\begin{align} \label{general}
\frac{1}{2} \diff \norm{\Lambda^s u}^{2}_{L^{2}}  +  &  \langle \Lambda^s ( \mathcal{Q} u) , \Lambda^s u  \rangle_{L^2} \diff W_t \nonumber \\
& = - \langle \Lambda^s (u\partial_{x}u), \Lambda^s u \rangle_{L^2} \diff t + \nu  \langle \Lambda^s (\partial_{xx} u), \Lambda^s u  \rangle_{L^2} \diff t \nonumber \\  
&+ \frac{1}{2}  \langle \Lambda^s (\mathcal{Q} u), \Lambda^s (\mathcal{Q} u)  \rangle_{L^2} \diff t + \frac{1}{2}  \langle \Lambda^s (\mathcal{Q}^2 u), \Lambda^s u  \rangle_{L^2} \diff t. 
\end{align}
Integrating by parts the first term of the right-hand side above and observing that $u \partial_x u = (1/2)\partial_x(u^2)$, we obtain
\begin{align*}
&-  \langle \Lambda^s (u\partial_{x}u), \Lambda^s u  \rangle_{L^2}  = \frac{1}{2}  \langle \Lambda^{s-1} (\partial_{x}(u^2)), \Lambda^{s+1} u  \rangle_{L^2}.
\end{align*}
Hence
\begin{align}
& | \langle \Lambda^s (u\partial_{x}u), \Lambda^s u  \rangle_{L^2} | \leq \norm{\Lambda^{s-1} (\partial_{x}(u^2))}_{L^2} \norm{\Lambda^{s+1} u}_{L^{2}} \nonumber \\
& \lesssim (\norm{u}_{L^\infty} \norm{\Lambda^{s} u}_{L^2}  )\norm{\Lambda^{s+1} u}_{L^2}  \leq \frac{1}{2\nu} \norm{u}_{L^\infty}^{2}\norm{\Lambda^s u}^{2}_{L^2} + \dfrac{\nu}{2} \norm{\Lambda^{s+1} u}^{2}_{L^{2}}, \label{visco1}
\end{align}
where we have employed Kato-Ponce (see Lemma \ref{katoponce}) in the second inequality. The second term on the right-hand side after the equality of equation \eqref{general} can be integrated as
\begin{align} \label{visco2}
\nu \langle \Lambda^s (\partial_{xx} u), \Lambda^{s} u  \rangle_{L^2} = -\nu \langle \Lambda^{s+1} u, \Lambda^{s+1} u \rangle_{L^2} = -\nu \norm{\Lambda^{s+1} u}^{2}_{L^2}. 
\end{align} 
As usual, applying inequality (\ref{eq:cancellation2:thm}) we also have the estimate
\begin{align}  \label{visco3}
| \langle \Lambda^s (\mathcal{Q} u), \Lambda^s (\mathcal{Q} u) \rangle_{L^2} +  \langle \Lambda^s (\mathcal{Q}^2 u), \Lambda^s u  \rangle_{L^2}  | \lesssim \norm{\Lambda^s u}^2.
\end{align}
Putting together \eqref{visco1}-\eqref{visco3}, one derives
\begin{align*} & \diff \norm{\Lambda^s u}^{2}_{L^{2}} + 2 \langle \Lambda^s (\mathcal{Q} u),  \Lambda^s u \rangle_{L^2}  \diff W_t + \nu \norm{\Lambda^{s+1} u}^{2}_{L^2} \diff t \\
& \hspace{170pt} \lesssim \frac{1}{\nu} (1 + \norm{u}_{L^\infty}^{2} )\norm{\Lambda^s u}^{2}_{L^2} \diff t. 
\end{align*}
By applying expectation in the above equation and taking into account that the expectation of the It\^o integral vanishes due to the martingale property, one obtains
\begin{align} \label{h1viscoso}
\mathbb{E} \left[ \sup_{t \in [0,T]} \norm{\Lambda^s u(t)}^{2}_{L^2} \right]  \lesssim & \mathbb{E} \left[ \norm{\Lambda^s u(0)}^{2}_{L^2} \right] \nonumber \\ 
&+ \frac{1}{\nu} \mathbb{E} \left[ \int_{0}^{T} 
(1+\norm{u(\gamma)}_{L^\infty}^{2} )\norm{\Lambda^s u(\gamma)}^{2}_{L^{2}} \diff \gamma \right],
\end{align}
and
\begin{align}\label{h2viscoso}
\nu \mathbb{E} \left[ \int_{0}^{T} \norm{\Lambda^{s+1} u(\gamma)}^{2}_{L^{2}} \diff \gamma \right]  \lesssim & \mathbb{E} \left[ \norm{\Lambda^s u(0)}^{2}_{L^2} \right] \nonumber \\
&+ \frac{1}{\nu} \mathbb{E} \left[ \int_{0}^{T} 
(1+\norm{u(\gamma)}_{L^\infty}^{2})\norm{\Lambda^s u(\gamma)}^{2}_{L^{2}} \diff \gamma \right] .
\end{align}
Now we claim that the following maximum principle holds 
\begin{align} \label{maximuminfty}
    \norm{u(t,\cdot)}_{L^\infty} \leq \norm{u(0,\cdot)}_{L^\infty}, \quad t \in [0,T],
\end{align}
which we show in Lemma \ref{max-principle}. Once \eqref{maximuminfty} is proven, we can infer from (\ref{h1viscoso}) and (\ref{h2viscoso}) together with Gr\"onwall's inequality that there exist constants $A_1 = A_1(T),$ $A_2 = A_2(T),$ such that 
\begin{align*} \mathbb{E} \left[ \sup_{t \in [0,T] } \norm{\Lambda^s u(t, \cdot)}^{2}_{L^2} \right] \leq A_1, \quad \nu\mathbb{E} \left[ \int_{0}^{T} \norm{\Lambda^{s+1} u(s, \cdot)}^{2}_{L^{2}} \diff s \right] \leq A_2. \end{align*}
This concludes the proof, so we are only left to show the maximum principle \eqref{maximuminfty}.
\end{proof}
\begin{lemma} \label{max-principle}
The maximum principle \eqref{maximuminfty} is satisfied. 
\end{lemma}

\begin{proof}
Following the type of argument in \cite{Beck2018}, consider the SDE
\begin{align} \label{difeo}
\diff X_t = \xi(X_t)\circ \diff W_t,
\end{align}
and let $\psi_t(X_0)$ be the corresponding flow of \eqref{difeo} with initial condition $X_0 \in \mathbb{T}$, which is a semimartingale and a smooth diffeomorphism. By applying It\^o-Wentzell in Stratonovich form (see Theorem \ref{Ito-Wentzell}), we evaluate $u$ along $X_t = \psi_t(X_0)$, getting that
\begin{align}
u(t,X_t) &= u(0,X_0) - \int_0^t u(s,X_s) \partial_x u (s,X_s) \diff s + \nu \int_0^t \partial_{xx} u(s,X_s) \diff s \nonumber \\ 
& - \int_0^t  \xi(X_s)\partial_x u (s,X_s) \circ \diff W_s + \int^t_0 \partial_x u (s,X_s) \circ \diff X_s, \label{something}
\end{align}
a.s., where the term in the last line 
\begin{align*}
    \int_0^t  \xi(X_s)\partial_x u (s,X_s) \circ \diff W_s + \int^t_0 \partial_x u (s,X_s) \circ \diff X_s = 0.
\end{align*} 
Consider the change of variables $w(t,X_0) = u(t, \psi_t(X_0))$ and by chain rule, obtain
\begin{align*}
    & \partial_x u (t,X_t) = \frac{\partial \psi_t^{-1}}{\partial x} (\psi_t(X_0)) \frac{\partial w}{\partial X_0}(t,X_0), \\
    &\partial_{xx} u (t,X_t) = \frac{\partial^2 \psi_t^{-1}}{\partial x^2} (\psi_t(X_0))\frac{\partial w}{\partial X_0}(t,X_0) + \left(\frac{\partial \psi_t^{-1}}{\partial x} (\psi_t(X_0))\right)^2 \frac{\partial^2 w}{\partial X_0^2}(t,X_0).
\end{align*}
Therefore, \eqref{something} is equivalent to the following PDE with random coefficients
\begin{align*}
    \frac{\partial w}{\partial t} + \Tilde{w} \frac{\partial w}{\partial X_0} = \Tilde{\nu} \frac{\partial^2 w}{\partial X_0^2}, 
\end{align*}
where
\begin{align*}
    &\Tilde{w} = \frac{\partial \psi_t^{-1}}{\partial x} (\psi_t(X_0)) u(t,\psi_t(X_0)) - \nu \frac{\partial^2 \psi_t^{-1}}{\partial x^2} (\psi_t(X_0)), \\
    &\Tilde{\nu} = \nu \left(\frac{\partial \psi_t^{-1}}{\partial x} (\psi_t(X_0))\right)^2.
\end{align*}
We claim that $w$ satisfies the maximum principle 
\begin{align} \label{ultimoprincipio}
\norm{w_t}_{L^\infty} \leq \norm{w_0}_{L^\infty}, \quad t >0.
\end{align}
To prove \eqref{ultimoprincipio}, we follow the strategy in \cite{antonio}. If $w_0=0,$ the claim is evident. Otherwise, for every $t\geq0,$ we define $X^{0}_{t} \in \mathbb{T}$ to be such that $\norm{w_t}_{L^\infty}=\lvert w(t,X^{0}_{t}) \rvert$ (i.e., the spatial point where $\norm{w_t}_{L^\infty}$ is attained. Its existence is guaranteed because $\mathbb{T}$ is compact). We start by claiming that for all $t>0,$ $\norm{w_t}_{L^\infty}$ is differentiable in time and
\begin{align} \label{claimclaim}
    \frac{\diff}{\diff t}\norm{w_t}_{L^\infty} = \frac{\partial w}{\partial t}(t,X^0_t).
\end{align}
First, we check that $\norm{w_t}_{L^\infty}$ is differentiable in time. For this, note that for $h \geq 0$ from the mean value theorem we have
\begin{align*}
& |\norm{w_{t+h}}_{L^\infty}-\norm{w_{t}}_{L^\infty}| \leq \norm{w_{t+h}-w_{t}}_{L^\infty} \leq \max_{s \in [t,t+h]}\lvert \partial_t w(s) \rvert h,
\end{align*}
and a similar argument can be applied if $h \leq 0.$ Therefore $\norm{w_t}_{L^\infty}$ is Lipschitz in time and from Rademacher's theorem a.e. differentiable.\\\\
We proceed to prove equality \eqref{claimclaim}. 
Without loss of generality, we assume $w(t,X_t^0) \geq 0,$ $t>0.$ We note that by taking small enough $h>0$
\begin{align*}
    & \frac{\diff}{\diff t} \norm{w_t}_{L^\infty} = \lim_{h \rightarrow 0} \frac{\norm{w_{t+h}}_{L^\infty} - \norm{w_t}_{L^\infty}}{h} = \lim_{h \rightarrow 0} \frac{w(t+h,X^{0}_{t+h}) - w(t,X^{0}_t)}{h} \\
    & = \lim_{h \rightarrow 0} \frac{w(t+h,X^{0}_{t+h}) - w(t+h,X^{0}_t)}{h} + \lim_{h \rightarrow 0} \frac{w(t+h,X^{0}_t) - w(t,X^{0}_t)}{h} \\
    &\hspace{250pt} \geq  \frac{\partial w}{\partial t}(t,X^0_t),
\end{align*}
for $t>0.$ Symmetrically, by choosing small enough $h>0$ we also have
\begin{align*}
    & \frac{\diff}{\diff t} \norm{w_t}_{L^\infty} = \lim_{h \rightarrow 0} \frac{\norm{w_{t-h}}_{L^\infty} - \norm{w_t}_{L^\infty}}{-h} = \lim_{h \rightarrow 0} \frac{w(t-h,X^{0}_{t-h}) - w(t,X^{0}_t)}{-h} \\
    & = \lim_{h \rightarrow 0} \frac{w(t-h,X^{0}_{t-h}) - w(t-h,X^{0}_t)}{-h} + \lim_{h \rightarrow 0} \frac{w(t-h,X^{0}_t) - w(t,X^{0}_t)}{-h} \\
    &\hspace{250pt} \leq  \frac{\partial w}{\partial t}(t,X^0_t),
\end{align*}
for $t>0.$ Therefore we conclude \eqref{claimclaim}. To finish our argument, we observe that 
\begin{align*}
    \frac{\partial w}{\partial t}(t,X^{0}_{t})=-\tilde{w} \frac{\partial w}{\partial X_0}(t,X^{0}_{t})+ \tilde{\nu} \frac{\partial^2 w}{\partial X_0^{2}}(t,X^{0}_{t}) \leq 0.
\end{align*} 
Hence \eqref{ultimoprincipio} is satisfied. Since $\psi_t$ is a diffeomorphism, we have $\norm{u_t}_{L^\infty} = \norm{w_t}_{L^\infty}$ so the maximum principle also follows for $u$.

 
\end{proof}
With this, we conclude the proof of Theorem \ref{mainvisc}.

\section{Conclusion and Outlook}\label{section5}
In this paper, we studied the solution properties of a stochastic Burgers' equation on the torus and the real line, with the noise appearing in the transport velocity. We have shown that this stochastic Burgers' equation is locally well-posed in $H^s(\mathbb{T},\mathbb{R}),$ for $s>2,$ and furthermore, established a blow-up criterion which extends the deterministic one to the stochastic case. We also proved that if the noise is of the form $\xi(x)\partial_x u \circ dW_t$ where $\xi(x) = \alpha x + \beta$, then shocks form almost surely from a negative slope. Moreover, for a more general type of noise, we showed that blow-up occurs in expectation, which follows from the previously mentioned stochastic blow-up criterion. Also, in the weak formulation of the problem, we provided a Rankine-Hugoniot type condition that is satisfied by the shocks, analogous to the deterministic shocks. Finally, we also studied the stochastic Burgers' equation with a viscous term, which we proved to be globally well-posed in $H^s$ for $s>2$.

Let us conclude by proposing some future research directions and open problems that have emerged during the course of this work:
\begin{itemize}
    \item Regarding shock formation, it is natural to ask whether our results can be extended to show that shock formation occurs almost surely for more general types of noise.
    \\
    \item Another possible question is whether our global well-posedness result can be extended for the viscous Burgers' equation with the Laplacian replaced by a fractional Laplacian $(-\Delta)^{\alpha},$ $\alpha \in (0,1)$. The main difficulty here is that in the stochastic case, the proof of the maximum principle (Proposition \eqref{max-principle}) does not follow immediately since the pointwise chain rule for the fractional Laplacian is not available. In the deterministic case, this question has been settled and it is known that the solution exhibits a very different behaviour depending on the value of $\alpha$: for $\alpha\in[1/2,1]$, the solution is global in time, and for $\alpha\in[0,1/2)$, the solution develops singularities in finite time \cite{KisNazSht2008,Kiselev2010}. Interestingly, when an It\^o noise of type $\beta u \diff W_t$ is added, it is shown in \cite{rockner2014local} that the probability of solutions blowing up for small initial conditions tends to zero when $\beta > 0$ is sufficiently large.
    It would be interesting to investigate whether the transport noise considered in this paper can also have a similar regularising effect on the equation.
    \\
    \item Similar results could be derived for other one-dimensional equations with non-local transport velocity \cite{CorCorFont2005,DeGregorio1990,DeGregorio1996}. For instance, the so called \textit{CCF model} \cite{CorCorFont2005} is also known to develop singularities in finite time, although by a different mechanism to that of Burgers'. To our knowledge, investigating these types of equations with transport noise is new.
\end{itemize}

\subsection*{Acknowledgements}
{\small
The authors would like to thank Jos\'e Antonio Carillo de la Plata, Dan Crisan, Theodore Drivas, Franco Flandoli, Darryl Holm, James-Michael Leahy, Erwin Luesink, and Wei Pan for encouraging comments and discussions that helped put together this work. DAO has been partially supported by the grant MTM2017-83496-P from the Spanish Ministry of Economy and Competitiveness, and through the Severo Ochoa Programme for Centres of Excellence in R\&D (SEV-2015-0554). ABdL has been supported by the Mathematics of Planet Earth Centre of Doctoral Training (MPE CDT).
ST acknowledges the Schr\"odinger scholarship scheme for funding during this work.
}

\bibliographystyle{alpha}
\bibliography{burger2018}

\newcommand{\etalchar}[1]{$^{#1}$}
\begin{thebibliography}{BFMM19}

\bibitem[AF09]{attanasio2009}
S.~Attanasio and F.~Flandoli.
\newblock Zero-noise solutions of linear transport equations without
  uniqueness: an example.
\newblock {\em Comptes Rendus Mathematique}, 347(13-14):753--756, 2009.

\bibitem[AOB20]{stochBouss2018}
D.~Alonso-Or\'{a}n and A.~Bethencourt.
\newblock On the {W}ell-posedness of stochastic {B}oussinesq equations with
  transport noise.
\newblock {\em Journal of Nonlinear Science}, 30(1):175--224, 2020.

\bibitem[AOHR21]{AlonHaoRohde}
D.~Alonso-Or\'an, T.~Hao, and C.~Rohde.
\newblock A local-in-time theory for singular {SDE}s with applications to fluid
  models with transport noise.
\newblock {\em Journal of Nonlinear Science}, 31:31--98, 2021.

\bibitem[BCJL94]{StoBurger1994}
L.~Bertini, N.~Cancrini, and G.~Jona-Lasinio.
\newblock The stochastic {B}urgers equation.
\newblock {\em Communications in Mathematical Physics 165, 211-231}, 1994.

\bibitem[BFMM19]{Beck2018}
L.~Beck, F.~Flandoli, Gubinelli M., and Maurelli M.
\newblock Stochastic {ODE}s and stochastic linear {PDE}s with critical drift:
  regularity, duality and uniqueness.
\newblock {\em Electronic Journal of Probability}, 24:1--72, 2019.

\bibitem[CC04]{antonio}
A.~C{\'o}rdoba and D.~C{\'o}rdoba.
\newblock A {M}aximum {P}rinciple {A}pplied to {Q}uasi-{G}eostrophic
  {E}quations.
\newblock {\em Communications in Mathematical Physics}, 249(249):511--528,
  2004.

\bibitem[CCF05]{CorCorFont2005}
A.~C\'ordoba, D.~C\'ordoba, and M.~Fontelos.
\newblock Formation of singularities for a transport equation with nonlocal
  velocity.
\newblock {\em Annals of Mathematics, 162, 1377-1389}, 2005.

\bibitem[CCH{\etalchar{+}}18]{cotter2018a}
C.~Cotter, D.~Crisan, D.~D. Holm, W.~Pan, and I.~Shevchenko.
\newblock Modelling uncertainty using circulation-preserving stochastic
  transport noise in a 2-layer quasi-geostrophic model.
\newblock {\em arXiv preprint arXiv:1802.05711}, 2018.

\bibitem[CCH{\etalchar{+}}19]{cotter2018b}
C.~J. Cotter, D.~Crisan, D.~D. Holm, W.~Pan, and I.~Shevchenko.
\newblock Numerically modelling stochastic {L}ie transport in fluid dynamics.
\newblock {\em Multiscale Modelling \& Simulation}, 17(1):192--232, 2019.

\bibitem[CFH19]{stochEuler2017}
D.~Crisan, F.~Flandoli, and Darryl~D. Holm.
\newblock Solution {P}roperties of a {3D} {S}tochastic {E}uler {F}luid
  {E}quation.
\newblock {\em Journal of {N}onlinear {S}cience}, 29(3):813--870, 2019.

\bibitem[CGH17]{homogenization2017}
C.~J. Cotter, G.~A. Gottwald, and D.~D. Holm.
\newblock Stochastic partial differential fluid equations as a diffusive limit
  of deterministic {L}agrangian multi-time dynamics.
\newblock {\em Proc. R. Soc. A}, 473(2205):20170388, 2017.

\bibitem[CH18]{stochasticCH}
D.~Crisan and D.~D. Holm.
\newblock Wave breaking for the stochastic {C}amassa--{H}olm equation.
\newblock {\em Physica D: Nonlinear Phenomena}, 376:138--143, 2018.

\bibitem[CO13]{4StoBurger1994}
P.~Catuogno and C.~Olivera.
\newblock Strong solution of the stochastic {B}urgers equation.
\newblock {\em Applicable Analysis: An International Journal, 93:3, 646-652,
  DOI: 10.1080/00036811.2013.797074}, 2013.

\bibitem[DFV14]{delarue2014}
F.~Delarue, F.~Flandoli, and D.~Vincenzi.
\newblock Noise prevents collapse of {V}lasov-{P}oisson point charges.
\newblock {\em Communications on Pure and Applied Mathematics},
  67(10):1700--1736, 2014.

\bibitem[DG90]{DeGregorio1990}
S.~De~Gregorio.
\newblock On a one-dimensional model for the three-dimensional vorticity
  equation,.
\newblock {\em J. Statist. Phys. 59, 1251-1263}, 1990.

\bibitem[DG96]{DeGregorio1996}
S.~De~Gregorio.
\newblock A partial differential equation arising in a {1}{D} model for the
  {3}{D} vorticity equa- tion.
\newblock {\em Math. Methods Appl. Sci. 19, 1233-1255.}, 1996.

\bibitem[DPDT94]{2StoBurger1994}
G.~Da~Prato, A.~Debussche, and R.~Temam.
\newblock Stochastic {B}urgers' equation.
\newblock {\em Nonlinear Differential Equations and Applications 1, 389-402},
  1994.

\bibitem[DPG07]{3StoBurger1994}
G.~Da~Prato and D.~Gatarek.
\newblock Stochastic {B}urgers equation with correlated noise.
\newblock {\em Stochastics and Stochastic Reports, 52:1-2, 29-41}, 2007.

\bibitem[FF13]{flandoli2013}
E.~Fedrizzi and F.~Flandoli.
\newblock Noise prevents singularities in linear transport equations.
\newblock {\em Journal of functional analysis}, 2013.

\bibitem[FG16]{friz2016stochastic}
P.~K. Friz and B.~Gess.
\newblock Stochastic scalar conservation laws driven by rough paths.
\newblock volume~33, pages 933--963. Elsevier, 2016.

\bibitem[FGH20]{funaki2019uniqueness}
T.~Funaki, Y.~Gao, and D.~Hilhorst.
\newblock Uniqueness of the entropy solution of a stochastic conservation law
  with a {Q}-brownian motion.
\newblock {\em Math Meth Appl Sci}, 43:5860--5886, 2020.

\bibitem[FGP10]{flandoli2010}
F.~Flandoli, M.~Gubinelli, and E.~Priola.
\newblock Well-posedness of the transport equation by stochastic perturbation.
\newblock {\em Inventiones mathematicae}, 2010.

\bibitem[FGP11]{flandoli2011}
F.~Flandoli, M.~Gubinelli, and E.~Priola.
\newblock Full well-posedness of point vortex dynamics corresponding to
  stochastic 2{D} {E}uler equations.
\newblock {\em Stochastic Processes and their Applications}, 2011.

\bibitem[FL18]{stoEuler2018}
F.~Flandoli and D.~Luo.
\newblock {E}uler-{L}agrangian approach to 3{D} stochastic {E}uler equations.
\newblock {\em arXiv:1803.05319}, 2018.

\bibitem[Fla11]{BookOfDoom}
F.~Flandoli.
\newblock {\em Random Perturbation of PDEs and Fluid Dynamic Models: {\'E}cole
  d’{\'e}t{\'e} de Probabilit{\'e}s de Saint-Flour XL--2010}, volume 2015.
\newblock Springer Science \& Business Media, 2011.

\bibitem[GHV14]{GlattVicol2014}
N.~Glatt-Holtz and V.~Vicol.
\newblock Local and global existence of smooth solutions for the stochastic
  {E}uler equation with multiplicative noise,.
\newblock {\em The Annals of Probability, Vol.42, 80-145,
  DOI:10.1214/12-AOP773}, 2014.

\bibitem[GM18]{gess2017}
B.~Gess and M.~Maurelli.
\newblock Well-posedness by noise for scalar conservation laws.
\newblock {\em Communications in Partial Differential Equations},
  43(12):1702--1736, 2018.

\bibitem[Gol85]{Goldstein85}
J.~A. Goldstein.
\newblock {\em Semigroups of linear operators and applications}.
\newblock Courier Dover Publications, 1985.

\bibitem[GS17]{gess2017stochastic}
B.~Gess and P.~E. Souganidis.
\newblock Stochastic non-isotropic degenerate parabolic--hyperbolic equations.
\newblock {\em Stochastic Processes and their Applications}, 127(9):2961--3004,
  2017.

\bibitem[Har93]{harrison1993}
J.~Harrison.
\newblock Stokes’ theorem for nonsmooth chains.
\newblock {\em Bulletin of the American Mathematical Society}, 29(2):235--242,
  1993.

\bibitem[Har99]{harrison1999}
J.~Harrison.
\newblock Flux across nonsmooth boundaries and fractal
  {G}auss/{G}reen/{S}tokes' theorems.
\newblock {\em Journal of Physics A}, 1999.

\bibitem[HN92]{harrison1992}
J.~Harrison and A.~Norton.
\newblock The {G}auss-{G}reen theorem for fractal boundaries.
\newblock {\em Duke Math. J}, 67(3):575--588, 1992.

\bibitem[HNS19]{hocquet2018}
A.~Hocquet, T.~Nilssen, and W.~Stannat.
\newblock Generalized {B}urgers' equation with rough transport noise.
\newblock {\em Stochastic Processes and their Applications}, 2019.

\bibitem[Hol15]{holm2015}
D.~D. Holm.
\newblock Variational principles for stochastic fluid dynamics.
\newblock {\em Proceedings of the Royal Society of London A: Mathematical,
  Physical and Engineering Sciences}, 471(2176), 2015.

\bibitem[Kis10]{Kiselev2010}
A.~Kiselev.
\newblock Regularity and blow up for active scalars.
\newblock {\em Math. Model. Nat. Phenom. Vol. 5, No. 4, 2010, pp. 225-255},
  2010.

\bibitem[KNS08]{KisNazSht2008}
A.~Kiselev, F.~Nazarov, and R.~Shterenberg.
\newblock Blow up and regularity for fractal {B}urgers equation.
\newblock {\em Dynamics of PDE, Vol.5, No.3, 211-240}, 2008.

\bibitem[KP88]{katoponce}
T.~Kato and G.~Ponce.
\newblock Commutator estimates and the {E}uler and {N}avier{-S}tokes equations.
\newblock {\em Communications in Pure and Applied Mathematics 41(7),
  891–907}, 1988.

\bibitem[KR81]{rozovskii}
N.~V. Krylov and B.~L. Rozovskii.
\newblock Stochastic evolution equations.
\newblock {\em Journal of Mathematical Sciences}, 16(249):1233--1277, 1981.

\bibitem[Kun81]{kunita81}
H.~Kunita.
\newblock Some extensions of {I}t{\^o}'s formula.
\newblock {\em S{\'e}minaire de probabilit{\'e}s}, 1981.

\bibitem[LY06]{lyons2006}
T.~J. Lyons and P.~SC Yam.
\newblock On {G}auss--{G}reen theorem and boundaries of a class of {H\"o}lder
  domains.
\newblock {\em Journal de math{\'e}matiques pures et appliqu{\'e}es},
  85(1):38--53, 2006.

\bibitem[RZZ14]{rockner2014local}
M.~R{\"o}ckner, R.~Zhu, and X.~Zhu.
\newblock Local existence and non-explosion of solutions for stochastic
  fractional partial differential equations driven by multiplicative noise.
\newblock {\em Stochastic Processes and their Applications}, 124(5):1974--2002,
  2014.

\bibitem[Sha57]{Shapiro1957}
V.~L. Shapiro.
\newblock The divergence theorem without differentiability conditions.
\newblock {\em Proceedings of the National Academy of Sciences of the United
  States of America}, 1957.

\bibitem[Sim86]{simon}
J.~Simon.
\newblock Compact sets in the space ${L}^p(0,{T};{B})$.
\newblock {\em Annali di Matematica Pura ed Applicata}, 146:65--96, 1986.

\bibitem[Ver81]{veretennikov1981}
A.~J. Veretennikov.
\newblock On strong solutions and explicit formulas for solutions of stochastic
  integral equations.
\newblock {\em Sbornik: Mathematics}, 39(3):387--403, 1981.

\end{thebibliography}

\end{document}